\documentclass{article}
\usepackage{latexsym}
\usepackage[varumlaut]{yfonts}
\usepackage{epsfig}
\usepackage{amsmath}
\usepackage{amssymb}
\usepackage{amsthm}
\newtheorem{theorem}{Theorem}[section]
\newtheorem{corollary}[theorem]{Corollary}

\newtheorem{lemma}[theorem]{Lemma}

\theoremstyle{remark}
\newtheorem{remark}[theorem]{Remark}

\theoremstyle{definition}
\newtheorem{definition}[theorem]{Definition}

\DeclareMathOperator\Aut{Aut}

\DeclareMathOperator\diag{diag}

\DeclareMathOperator\tr{tr}
\DeclareMathOperator\spa{span}

\begin{document}

\title{Graph immersions with parallel cubic form}

\author{Roland Hildebrand \thanks{%
Univ.\ Grenoble Alpes, CNRS, Grenoble INP, LJK, 38000 Grenoble, France
({\tt roland.hildebrand@univ-grenoble-alpes.fr}).}}

\maketitle

\begin{abstract}
We consider non-degenerate graph immersions into affine space $\mathbb A^{n+1}$ whose cubic form is parallel with respect to the Levi-Civita connection of the affine metric. There exists a correspondence between such graph immersions and pairs $(J,\gamma)$, where $J$ is an $n$-dimensional real Jordan algebra and $\gamma$ is a non-degenerate trace form on $J$. Every graph immersion with parallel cubic form can be extended to an affine complete symmetric space covering the maximal connected component of zero in the set of quasi-regular elements in the algebra $J$. It is an improper affine hypersphere if and only if the corresponding Jordan algebra is nilpotent. In this case it is an affine complete, Euclidean complete graph immersion, with a polynomial as globally defining function. We classify all such hyperspheres up to dimension 5. As a special case we describe a connection between Cayley hypersurfaces and polynomial quotient algebras. Our algebraic approach can be used to study also other classes of hypersurfaces with parallel cubic form.
\end{abstract}



\section{Introduction}

The cubic form $C$ of an equiaffine hypersurface immersion is the covariant derivative of the affine metric $h$ with respect to the affine connection $\nabla$. Affine hypersurface immersions with parallel cubic form have been studied in various settings for 30 years, and their classification is an important problem in affine differential geometry.

One can consider immersions whose cubic form is parallel with respect to the affine connection, $\nabla C = 0$. Non-degenerate Blaschke immersions satisfying this condition are either quadrics or graph immersions whose defining function is a cubic polynomial \cite{Vrancken88}. Actually, in the latter case the immersion must be an improper affine hypersphere \cite{BNS90}, and the determinant of the Hessian of the defining function identically equals $\pm1$ \cite[Example 3.3, pp.~47--48]{NomizuSasaki}. Non-degenerate Blaschke hypersurface immersions with $\nabla C = 0$ into $\mathbb R^k$, $k = 3,4,5,6$, have been classified in \cite{NomizuPinkall89},\cite{Vrancken88},\cite{Gigena02},\cite{Gigena03}, respectively. In \cite{Gigena11} an algorithm was presented to classify all such immersions for a given arbitrary dimension.

Another class of hypersurface immersions is obtained when the condition $\nabla K = 0$ is imposed, where $K = \nabla - \hat\nabla$ is the difference tensor between the affine connection and the Levi-Civita connection of the affine metric. Non-degenerate Blaschke immersions with this property have been studied in \cite{DillenVrancken98}. There it was established that, as for $\nabla C = 0$, they are either quadrics or improper affine hyperspheres. In the latter case the defining function is given by a polynomial, the affine metric is flat, the difference tensor is nilpotent, i.e., $K_X^m = 0$ for some $m > 1$ and all vector fields $X$, and $[K_X,K_Y] = 0$ for all vector fields $X,Y$. A classification of these improper affine hyperspheres has been obtained for several special cases, namely, if $K^2 = 0$, if the affine metric has Lorentzian signature, and if $K^{n-2} \not= 0$.

Parallelism of the cubic form can also be defined with respect to the connection $\hat\nabla$. Since the affine metric is parallel with respect to $\hat\nabla$, the conditions $\hat\nabla C = 0$ and $\hat\nabla K = 0$ are equivalent. Much work concentrated on the case of Blaschke immersions. A Blaschke immersion satisfying $\hat\nabla C = 0$ must be an affine hypersphere \cite{BNS90}. This hypersphere may be proper or improper.

An $n$-dimensional improper affine hypersphere satisfying $\hat\nabla C = 0$ and which is not a quadric must be affinely equivalent to the {\it Cayley surface} $z = xy-\frac13y^3$ for $n = 2$ \cite{MagidNomizu89}, to one of the graph immersions
\begin{equation} \label{class3}
w = xy+\frac12z^2-\frac13y^3, \qquad w = xy+\frac12z^2-zy^2+\frac14y^4
\end{equation}
for $n = 3$ \cite{HuLi11}, and to one of the graph immersions
\begin{equation} \label{class4}
\begin{aligned}
x_5 &= x_1x_2+x_3x_4-\frac13(x_1^2+x_3^2), \\
x_5 &= x_1x_2+x_3x_4-\frac13x_1x_3(x_1+x_3), \\
x_5 &= x_1x_2+x_3x_4-\frac13x_1x_3^2, \\
x_5 &= x_1x_3-x_1^2x_2+\frac14x_1^4+\frac12(x_2^2-x_4^2), \\
x_5 &= x_1x_2-\frac13x_1^3+\frac12(x_3^2-x_4^2), \\
x_5 &= x_1x_4+x_2x_3-x_1^2x_3-x_1x_2^2+x_1^3x_2-\frac15x_1^5, \\
x_5 &= x_1x_2-\frac13x_1^3+\frac12(x_3^2+x_4^2), \\
x_5 &= x_1x_3+\frac12x_2^2-x_1^2x_2+\frac14x_1^4+\frac12x_4^2
\end{aligned}
\end{equation}
for $n = 4$ \cite{HLLV11b}. For general $n \geq 3$ it cannot be convex \cite{HLV11}, and if it has a Lorentzian metric, then it must be affinely equivalent to one of the graph immersions \cite{HLLV11a}
\[ x_{n+1} = x_1x_2-\frac13x_1^3+\frac12\sum_{k=3}^nx_k^2, \qquad x_{n+1} = x_1x_3+\frac12x_2^2-x_1^2x_2+\frac14x_1^4+\frac12\sum_{k=4}^nx_k^2.
\]
In \cite{CeceLi14} the case of vanishing Pick invariant, constant sectional curvature, and metric with negative index 2 has been classified. In \cite[Sect.\ 6.1.3]{Hildebrand15a} we have sketched a connection between improper affine hyperspheres satisfying $\hat\nabla C = 0$ and nilpotent Jordan algebras.

The second hypersurface in \eqref{class3} and the sixth hypersurface in \eqref{class4} are generalizations of the Cayley surface known as {\it Cayley hypersurfaces}. These hypersurfaces were introduced in \cite{EE06} and shown to be improper hyperspheres with vanishing Pick invariant, satisfying $C \not= 0$ and $\hat\nabla C = 0$ \cite{HLZ11}. Further generalizations of these hypersurfaces have been shown in \cite{LiZhang15} to be characterized by a parallelism condition involving connections which are affine combinations of $\nabla$ and $\hat\nabla$.

\medskip

The main subject of this paper are graph immersions into affine space $\mathbb A^{n+1}$ satisfying the condition $\hat\nabla C = 0$. We establish a correspondence between this class of hypersurfaces and the class of Jordan algebras with non-degenerate trace form, so-called {\it metrised} Jordan algebras \cite{Bordemann97}. We show that such graph immersions can be extended to affine complete symmetric spaces, which are translation-invariant coverings of the quasi-regular domain containing zero in the corresponding Jordan algebra.

The structure of the present paper closely matches that of \cite{Hildebrand15a} on centro-affine immersions with parallel cubic form, but the technical details differ in several respects. The main difference is that here we consider general real Jordan algebras instead of unital ones. The role of the unit element in \cite{Hildebrand15a} is played by the zero element here, and the inverse is replaced by the quasi-inverse. While in \cite{Hildebrand15a} the proper affine hyperspheres satisfying the condition $\hat\nabla C = 0$ correspond to the subclass of semi-simple Jordan algebras, here the improper affine hyperspheres satisfying $\hat\nabla C = 0$ correspond to the nilpotent Jordan algebras.

In particular, the classification of such improper hyperspheres reduces to the classification of nilpotent metrised Jordan algebras. In contrast to the semi-simple Jordan algebras, the nilpotent ones are not fully classified. Partial results are available in \cite{GM75},\cite{ES02a},\cite{ES02b}, where nilpotent Jordan algebras of dimensions $\geq 4$,5, and 6, respectively, have been classified. In \cite{CS99},\cite{ES02a} nilpotent Jordan algebras of dimension $n$ and nilindex $n,n-1$, respectively, have been classified. These classifications are not always up to isomorphism, however. From \cite[Theorem 3.1]{Bordemann97} it follows that the classification of metrised nilpotent Jordan algebras is not simpler than that of nilpotent Jordan algebras. A classification of improper affine hyperspheres satisfying $\hat\nabla C = 0$ is thus out of reach. We will show, however, that such improper affine hyperspheres must be graph immersions with a polynomial as global defining function.

In \cite{DillenVrancken94},\cite{HLV08} it has been observed that the Calabi product of proper affine hyperspheres with parallel cubic form or of such a hypersphere with a point are again proper affine hyperspheres with parallel cubic form, and hence one can speak of decomposable or irreducible such immersions. In a classification, one then only needs to consider the irreducible immersions. We will provide a similar notion of decomposability and irreducibility for graph immersions with parallel cubic form. This notion also respects the property of being an improper affine hypersphere.

The remainder of the paper is structured as follows. In the next section we review the notion of graph immersions, in particular, we define the notion of product and irreducibility for graph immersions in Subsection \ref{subs_prod}. In Section \ref{sec_Jordan} we provide the necessary background on Jordan algebras. Section \ref{sec_correspondence} contains the main technical results of the paper, namely how exactly metrised Jordan algebras are related to graph immersions with parallel cubic form. In Section \ref{sec_classification} we consider specifically improper affine hyperspheres with parallel cubic form and their algebraic counterparts, the nilpotent metrised Jordan algebras. We provide a structural result in Theorem \ref{th_can} and show in Corollary \ref{AHS_sign_dim} that the dimension of an irreducible improper hypersphere with parallel cubic form is bounded by a quadratic function of the number of negative eigenvalues of its affine fundamental form. In Subsection \ref{subs_Cayley} we consider the Cayley hypersurfaces more closely. In Subsection \ref{subs_lowdim} we provide a full classification of irreducible improper affine hyperspheres with parallel cubic form up to dimension 5. In Sections \ref{sec_desc}, \ref{sec_correspondence}, and \ref{sec_classification}, whenever we speak about {\it parallel cubic form}, we will mean the condition $\hat\nabla C = 0$. In Section \ref{sec_other} we apply our methods to other classes of hypersurface immersions with parallel cubic form, in particular, those mentioned at the beginning of this introduction.

\section{Graph immersions} \label{sec_desc}

In this section we review some basic facts about graph immersions and improper affine hyperspheres. We introduce a notion of products of graph immersions which is similar to the Calabi product of proper affine hyperspheres.

\medskip

A graph immersion $f: M_n \to \mathbb A^{n+1}$ of an $n$-dimensional manifold into $(n+1)$-dimensional affine space is an affine hypersurface immersion equipped with a parallel transversal vector field $\xi$. It is non-degenerate if its affine fundamental form $h$ is non-degenerate. The induced connection $\nabla$ of a graph immersion is flat \cite[Example 2.4]{NomizuSasaki}, and we may locally introduce an affine coordinate system on $M_n$. On a simply connected affine chart there exists a scalar function $F$ such that the immersion $f$ is locally affinely equivalent to the graph immersion $x \mapsto (x,F(x))$ with transversal vector field $({\bf 0},1)$ \cite[Prop.~2.8]{NomizuSasaki}. We shall call $F$ a local \emph{representing function} of the immersion $f$. Here and throughout the paper ${\bf 0}$ denotes the zero vector in an $n$-dimensional vector space, which may be a tangent space to a graph immersion or a Jordan algebra.

When considering local properties of the graph immersion, it will be convenient to define the immersion $f$ directly via its representing function. Let $\Omega \subset \mathbb R^n$ be a domain and $F: \Omega \to \mathbb R$ a smooth function. The graph immersion $f: \Omega \to \mathbb R^{n+1} = \mathbb R^n \times \mathbb R$ represented by $F$ is then defined by $f(y) = (y,F(y))$, and the transversal vector field is given by $\xi = ({\bf 0},1)$. The induced connection $\nabla$ equals the canonical flat affine connection on $\mathbb R^n$. The affine fundamental form $h$ is given by the derivative $\nabla^2F = F''$ \cite[p.~39]{NomizuSasaki}. In the sequel we will assume that $h$ is {\it non-degenerate}, in which case it defines a Hessian pseudo-metric on $\Omega$. We shall use this metric to raise and lower indices of tensors on $\Omega$.



From the point of view of affine differential geometry, neither the decomposition $\mathbb R^{n+1} = \mathbb R^n \times \mathbb R$ nor the vector space structure of the target space $\mathbb R^{n+1}$, and hence the representing function $F$, have an invariant meaning. Indeed, the natural target space of an $n$-dimensional affine hypersurface immersion is the affine space $\mathbb A^{n+1}$. The affine connection as well as the affine fundamental form on $\Omega$ are determined solely by the map $f: \Omega \to \mathbb A^{n+1}$ and the transversal vector field $\xi$. Given these objects, a representing function $F$ of the graph immersion can be obtained by specifying an affine hyperplane $A \subset \mathbb A^{n+1}$ which is transversal to $\xi$, and considering $\mathbb A^{n+1}$ as a direct sum $A + \spa\xi$, corresponding to the decomposition $\mathbb R^{n+1} = \mathbb R^n \times \mathbb R$. For $y \in \Omega$ we may then set $F(y) = \alpha$, where $\alpha$ is the unique real number such that $f(y) - \alpha\xi \in A$. The affine connection $\nabla$ is then given by the pullback of the flat affine connection on $A$ to $\Omega$ by virtue of the injective local diffeomorphism $\iota: \Omega \to A$ defined by $\iota: y \mapsto f(y) - F(y)\xi = (f(y) + \spa\xi) \cap A$.

For a given $y \in \Omega$, there exists a unique affine hyperplane $A$ such that $f(y) \in A$ and $A$ is tangent to the image $f[V]$ at $f(y)$, where $V \subset \Omega$ is a small enough neighbourhood of $y$. The corresponding representing function then satisfies $F(y) = 0$, $F'(y) = 0$. Moreover, in this case we can naturally identify $A$ with the tangent space $T_y\Omega$, by putting into correspondence the tangent vector $v \in T_y\Omega$ with the point $f(y)+f_*(v) \in A$, where $f_*$ is the differential of $f$ at $y$. Then the product $T_y\Omega \times \mathbb R$ can be identified with $\mathbb A^{n+1}$ by the relation
\begin{equation} \label{identification}
(v,\alpha) \mapsto f(y)+f_*(v)+\alpha\xi.
\end{equation}
We will make use of this identification later in Subsection \ref{sec4_3}.

\medskip

We adopt the Einstein summation convention over repeating indices. Denote the derivatives of $F$ with respect to the affine connection $\nabla$ by indices after a comma. Thus we have $\nabla_{\alpha}F = F_{,\alpha}$, $\nabla_{\alpha}\nabla_{\beta}F = F_{,\alpha\beta}$ etc. Denote the elements of the inverse of the Hessian $F'' = \nabla^2F$ by $F^{,\alpha\beta}$.

Let $\hat\nabla$ be the Levi-Civita connection of the affine pseudo-metric $h$, $C = \nabla h$ the cubic form, and $K = \nabla - \hat\nabla$ the difference tensor. These affine invariants can be represented by expressions depending on the derivatives of $F$, as follows:
\begin{equation} \label{correspondence}
h_{\alpha\beta} = F_{,\alpha\beta}, \quad C_{\alpha\beta\gamma} = F_{,\alpha\beta\gamma}, \quad K^{\gamma}_{\alpha\beta} = -\frac12F_{,\alpha\beta\delta}F^{,\gamma\delta}.
\end{equation}
The covariant derivative of $C$ with respect to $\hat\nabla$ is given by
\begin{equation} \label{parallel_F3}
\hat\nabla_{\delta}C_{\alpha\beta\gamma} = F_{,\alpha\beta\gamma\delta} - \frac12F^{,\rho\sigma}(F_{,\alpha\beta\rho}F_{,\gamma\sigma\delta} + F_{,\alpha\gamma\rho}F_{,\beta\sigma\delta} + F_{,\beta\gamma\rho}F_{,\alpha\sigma\delta}).
\end{equation}

{\corollary \label{surface_chars} Let $f: \Omega \to \mathbb A^{n+1}$ be a smooth non-degenerate graph immersion with representing function $F: \Omega \to \mathbb R$. Then the immersion $f$ has parallel cubic form with respect to the Levi-Civita connection of the affine metric if and only if
\begin{equation} \label{quasi_lin_PDE}
F_{,\alpha\beta\gamma\delta} = \frac12F^{,\rho\sigma}(F_{,\alpha\beta\rho}F_{,\gamma\delta\sigma} + F_{,\alpha\gamma\rho}F_{,\beta\delta\sigma} + F_{,\alpha\delta\rho}F_{,\beta\gamma\sigma}).
\end{equation} }

\begin{proof}
The corollary is a direct consequence of \eqref{parallel_F3}.
\end{proof}

We have the following characterization of improper affine hyperspheres.

{\lemma \label{IAHS} \cite[Example 3.3, pp.~47--48]{NomizuSasaki} Let $f: \Omega \to \mathbb A^{n+1}$ be a smooth non-degenerate graph immersion with representing function $F: \Omega \to \mathbb R$. Then $f$ is an improper affine hypersphere with affine normal equal to $\xi$ if and only if $\det F'' = \pm1$. }

\subsection{Products of graph immersions} \label{subs_prod}

In this subsection we consider a notion of product for graph immersions and improper affine hyperspheres.

\begin{definition} \label{def_prod}
Let $\Omega_k \subset \mathbb R^{n_k}$, $k = 1,\dots,r$, be domains, $F_k: \Omega_k \to \mathbb R$ smooth functions, and $f_k: \Omega_k \to \mathbb R^{n_k+1}$ the graph immersions defined by $F_k$. Set $\Omega = \Omega_1 \times \dots \times \Omega_r$ and define the function $F: \Omega \to \mathbb R$ by $F(x_1,\dots,x_r) = F_1(x_1) + \dots + F_r(x_r)$. We shall call the graph immersion $f: \Omega \to \mathbb R^{n_1+\dots+n_r+1}$ defined by the function $F$ the {\sl product} of the immersions $f_1,\dots,f_r$. A graph immersion which can be in a nontrivial way (locally) represented as such a product will be called {\sl (locally) decomposable}, otherwise it will be called {\sl irreducible}.
\end{definition}

\begin{lemma} \label{lem_prod}
Assume the notations of Definition \ref{def_prod}. If the graph immersions $f_1,\dots,f_r$ are non-degenerate, then $f$ is non-degenerate. If $f_1,\dots,f_r$ are improper affine hyperspheres, then $f$ is an improper affine hypersphere. If $f_1,\dots,f_r$ have parallel cubic form, then $f$ has parallel cubic form.
\end{lemma}

\begin{proof}
The mixed derivatives $\frac{\partial^2F}{\partial x_k\partial x_l}$, $k \not= l$, are identically zero, and hence all derivatives of $F$ will be block-diagonal, with the corresponding derivatives of $F_k$ in the diagonal blocks. In particular, we have $\det F'' = \prod_{k=1}^r \det F_k''$. Hence $\det F_k'' \not= 0$ implies $\det F'' \not= 0$, which yields the first assertion of the lemma. The conditions $\det F_k'' = \pm1$ imply $\det F'' = \pm1$, which by Lemma \ref{IAHS} yields the second assertion of the lemma. The third assertion follows from Corollary \ref{surface_chars} and the block-diagonal structure of the tensors $F^{,\rho\sigma},F_{,\alpha\beta\gamma},F_{,\alpha\beta\gamma\delta}$.
\end{proof}

\section{Jordan algebras} \label{sec_Jordan}

In this section we provide the necessary background on Jordan algebras. Much of the material in this section is taken from \cite[Section 14]{ZSSS82}. Other references on Jordan algebras are \cite{Jacobson68},\cite{Koecher99}, or \cite{McCrimmon}.

{\definition \label{Jordan_def} \cite[p.~3]{McCrimmon} A real {\sl Jordan algebra} $J$ is a real vector space endowed with a bilinear operation $\bullet: J \times J \to J$ satisfying the following conditions:

i) commutativity: $x \bullet y = y \bullet x$ for all $x,y \in J$,

ii) Jordan identity: $x \bullet (x^2 \bullet y) = x^2 \bullet (x \bullet y)$ for all $x,y \in J$, where $x^2 = x \bullet x$.
}

Throughout the paper we assume that $J$ is finite-dimensional. Define $x^{k+1} = x \bullet x^k$ recursively for all $x \in J$ and $k \geq 1$.

\medskip

Let us denote the operator of multiplication with the element $x$ by $L_x$, $L_xy = x \bullet y = L_yx$. Then the Jordan identity can be written as $[L_x,L_{x^2}] = 0$. Jordan algebras are power-associative, i.e., $x^k \bullet x^l = x^{k+l}$ for all $k,l \geq 1$. Moreover, we have the following result.

{\lemma \cite[p.~35]{Jacobson68} \label{strongly_ass} In a Jordan algebra we have $[L_{x^k},L_{x^l}] = 0$ for all $k,l \geq 1$, and all operators $L_{x^k}$ are generated by the two operators $L_x,L_{x^2}$. }

{\definition \cite[p.~24]{Schaefer} \label{def_trace_form} A {\sl trace form} on a Jordan algebra $J$ is a symmetric bilinear form $\gamma$ such that
\begin{equation} \label{3symmetry}
\gamma(u \bullet v,w) = \gamma(u,v \bullet w)
\end{equation}
for all $u,v,w \in J$. Equivalently, the operator $L_v$ is self-adjoint with respect to $\gamma$ for all $v \in J$. }

{\definition \cite{Bordemann97} A pair $(J,\gamma)$, where $J$ is a Jordan algebra and $\gamma$ a non-degenerate trace form on $J$, is called {\sl metrised}. }


{\definition \cite[p.~58]{Koecher99} An element $u$ of a Jordan algebra $J$ is called {\sl nilpotent} if $u^r = 0$ for some $r \geq 1$. }

{\lemma \cite[Theorem III.5, p.~58]{Koecher99} \label{niltrace} For an element $u$ of a Jordan algebra $J$ the following are equivalent:

\begin{itemize}
\item $u$ is nilpotent,
\item $L_u$ is nilpotent,
\item there exists $m > 0$ such that $\tr L_{u^r} = 0$ for all $r \geq m$.
\end{itemize} }

{\definition \cite[p.~195]{Jacobson68} A Jordan algebra $J$ is called {\sl nilalgebra} if it consists of nilpotent elements. It is called {\sl nilpotent} if there exists $m \geq 1$ such that the product of every $m$ elements is zero. }

Now every Jordan nilalgebra $J$ is solvable \cite[Theorem 3, p.~196]{Jacobson68} and every solvable Jordan algebra is nilpotent \cite[Cor.\ 1, p.~195]{Jacobson68}. We thus have the following lemma.

{\lemma \label{nilnil} Let $J$ be a Jordan nilalgebra. Then $J$ is nilpotent. }

{\definition A Jordan algebra $J$ is called {\sl direct sum} of the subalgebras $J_1,\dots,J_r$, $J = \oplus_{k = 1}^r J_k$, if $J$ is the sum of $J_1,\dots,J_r$ as a vector space and $x \bullet y = 0$ for all $x \in J_k$, $y \in J_l$ with $k \not= l$. }

Clearly the summands in a direct sum decomposition are ideals, i.e., $x \bullet y \in J_k$ for all $x \in J_k$, $y \in J$.

{\definition A metrised Jordan algebra $(J,\gamma)$ is called {\sl direct sum} of the metrised subalgebras $(J_1,\gamma_1),\dots,(J_r,\gamma_r)$ if $J = \oplus_{k = 1}^r J_k$, $\gamma_k = \gamma|_{J_k}$ for all $k$, and the ideals $J_k$ are mutually $\gamma$-orthogonal. }


An element $e \in J$ is called {\sl unit element} if $x \bullet e = x$ for all $x \in J$. A Jordan algebra possessing a unit element is called {\sl unital}. In a unital Jordan algebra $J$, $y = x^{-1}$ is called the {\sl inverse} of $x$ if $x \bullet y = e$ and $L_x,L_y$ commute \cite[p.~67]{Koecher99}. If it exists, the inverse is unique \cite[Lemma III.5, p.~66]{Koecher99} and satisfies $(x^{-1})^{-1} = x$ \cite[p.~67]{Koecher99}. In this case we shall call $x$ {\sl invertible}. Introduce the operator $U_x = 2L_x^2 - L_{x^2}$, which is quadratic in the parameter $x$. We have the following characterization of the invertible elements of $J$.

{\theorem \cite[Theorem III.12, p.~67]{Koecher99} \label{th_invertible} Let $J$ be a unital Jordan algebra. An element $x \in J$ is invertible if and only if $\det U_x \not= 0$. In this case the inverse is given by $x^{-1} = U_x^{-1}x$, and
\[ U_{x^{-1}} = U_x^{-1},\qquad L_{x^{-1}} = L_xU^{-1}_x = U^{-1}_xL_x.
\] }

It follows that the set of invertible elements in a unital Jordan algebra is open, dense, and conic, its complement being the zero set of the homogeneous polynomial $\det U_x$.

The derivative of the inverse is given by \cite[eq.~(1), p.~73]{Koecher99}
\begin{equation} \label{inv_der}
D_ux^{-1} = -U_x^{-1}u.
\end{equation}
Here $D_u$ denotes the derivative with respect to $x$ in the direction of $u$.

{\definition \cite[p.~52]{McCrimmon} Let $J$ be a Jordan algebra with multiplication $\bullet$. The {\sl unital hull} $\hat J$ of $J$ is the vector space $\mathbb R \times J$ equipped with a multiplication $\hat\bullet$ defined by
\begin{equation} \label{def_hat_prod}
(\alpha,u) \hat\bullet (\beta,v) = (\alpha\beta,\alpha v + \beta u + u \bullet v).
\end{equation} }

The unital hull $\hat J$ is a unital Jordan algebra, with $\hat e = (1,{\bf 0})$ as its unit element, and the original algebra $J$ is canonically isomorphic to the ideal $\{ 0 \} \times J \subset \hat J$ \cite[p.~149]{McCrimmon}.

{\definition \cite[p.~307]{ZSSS82} Let $J$ be a Jordan algebra. An element $a \in J$ is called {\sl quasi-regular} if there exists $b \in J$ such that in the unital hull $\hat J$ of $J$ the element $(1,-a)$ is invertible and $(1,-a)^{-1} = (1,-b)$. Then $b$ is called {\sl quasi-inverse} and denoted by $a^{(-1)}$. }

{\lemma \cite[Theorem 7, pp.~307--308]{ZSSS82} \label{quasi_char} Let $J$ be a Jordan algebra. An element $a \in J$ is quasi-regular with $b \in J$ its quasi-inverse if and only if $(1,-a) \hat\bullet (1,-b) = \hat e$ and $[L_a,L_b] = 0$. }

Note that the zero element is always quasi-regular, and its quasi-inverse is again zero. Moreover, if $(1,-a)$ has an inverse in $\hat J$, then by \eqref{def_hat_prod} the first component of this inverse must be equal to 1. Thus $a$ is quasi-regular if and only if $(1,-a)$ is invertible in $\hat J$. The set of quasi-regular elements is hence open and dense in $J$.

In the present paper, a central role is played by domains of elements $a \in J$ such that $(1,a)$ is invertible in $\hat J$, i.e., elements whose negative is quasi-regular. We shall denote the connected component of zero in the set of all such elements by ${\cal Y}$. Hence $a \in {\cal Y}$ if and only if $-a$ is quasi-regular and there exists a path linking $-a$ and $0$ in the set of quasi-regular elements. For $a \in J$ such that $-a$ is quasi-regular we have by definition that $(1,a) \hat\bullet (1,-(-a)^{(-1)}) = (1,0)$, which by \eqref{def_hat_prod} yields
\begin{equation} \label{quasi_prod}
a \bullet (-a)^{(-1)} = a - (-a)^{(-1)}.
\end{equation}

{\lemma \cite[p.~310]{ZSSS82} \label{nilY} Let $J$ be a nilpotent Jordan algebra. Then every element of $J$ is quasi-regular. }

\subsection{The set of quasi-regular elements} \label{sec_inv}

In this subsection we shall investigate the symmetry properties of the set of quasi-regular elements of a Jordan algebra, more precisely of the connected component ${\cal Y}$ of the zero element in this set. It will turn out that any graph immersion with parallel cubic form is a covering space, for which the component ${\cal Y}$ serves as a base. The main goal will be to represent ${\cal Y}$ as a homogeneous space and to describe its symmetry group. The second goal will be to compute the derivative of the quasi-inverse.

Let $J$ be a real Jordan algebra, $\hat J$ its unital hull, and $\hat e \in \hat J$ the unit element. Denote by $\hat{\cal Y}$ the connected component of $\hat e$ in the set of invertible elements of $\hat J$. Since $\hat{\cal Y}$ is conic, we have that the connected component ${\cal Y}$ of zero in the set of elements whose negative is quasi-regular is given by ${\cal Y} = \{ a \in J \,|\, (1,a) \in \hat{\cal Y} \}$, i.e., $\{1\} \times {\cal Y} = \hat{\cal Y} \cap (\{1\} \times J)$.

Let $\Pi \subset GL(\hat J)$ be the group generated by the transformations $U_{\hat u}$, where $\hat u \in \hat J$ varies in a small neighbourhood of $\hat e$. Define also the subgroup $\Pi_1 \subset \Pi$ generated by the transformations $U_{(1,u)}$, where $u \in J$ varies in a small neighbourhood of zero. For any $a > 0$, $u \in J$ we have $U_{(a,au)} = a^2U_{(1,u)}$. Hence $\Pi = \Pi_1 \times \mathbb R_{++}$ decomposes into a direct product.

{\theorem \cite[Theorem VI.2, p.~110]{Koecher99} \label{transitive_action} The group $\Pi$ acts transitively on $\hat{\cal Y}$. }

In the context of \cite[Chapter VI]{Koecher99} this theorem was proven for semi-simple Jordan algebras. However, as we have sketched in \cite[Section 3.1]{Hildebrand15a}, the proof is valid for every unital Jordan algebra, in particular, for $\hat J$.

For any $\hat u = (a,u) \in \hat J$ the operator $U_{\hat u}$ multiplies the first component of $\hat J = \mathbb R \times J$ by $a^2$. Therefore $\Pi_1$ is precisely the subgroup of $\Pi$ which leaves the set $\{1\} \times {\cal Y} \subset \hat{\cal Y}$ invariant.

Let us find the generators of $\Pi_1$. Set $\hat u^0 = \hat e$ for every $\hat u \in \hat J$. Then we can define the {\sl exponential} $\exp(\hat u) = \sum_{k=0}^{\infty} \frac{1}{k!} \hat u^k = \exp(L_{\hat u})\hat e$, which bijectively maps a neighbourhood of zero in $\hat J$ to a neighbourhood of $\hat e$ \cite[pp.~82--83]{Koecher99}. Note that for $\hat u = (a,u)$, the first component of $\exp(\hat u)$ is given by $e^a$.

{\lemma \cite[Lemma IV.4, p.~83]{Koecher99} For every $\hat w \in \hat J$ we have $U_{\exp(\hat w)} = \exp(2L_{\hat w})$. }

Therefore the group $\Pi$ is generated by the 1-dimensional subgroups $\exp(tL_{\hat w})$, $\hat w \in \hat J$, and the group $\Pi_1$ is generated by the 1-dimensional subgroups $\exp(tL_{(0,w)})$, $w \in J$. The latter generate a flow $\varphi_w: J \times \mathbb R \to J$ with tangent vector field
\begin{equation} \label{flow_gen}
X_w(x) = w + w \bullet x = w + L_wx.
\end{equation}
Note that this vector field is affine-linear, and hence $\Pi_1$ acts on $J$ by affine-linear transformations.

{\corollary \label{cor:Pi1transitive} The group $\Pi_1$ acts transitively on ${\cal Y}$. }

\begin{proof}
By Theorem \ref{transitive_action} $\Pi_1$ acts transitively on the set $\{1\} \times {\cal Y} \subset \hat{\cal Y}$. This action induces a transitive action also on ${\cal Y} \subset J$.
\end{proof}

%
%

\medskip

We shall now compute the $L$ and $U$ operators in $\hat J$. By \eqref{def_hat_prod} we have for all $x \in J$, $\hat x = (1,x)$
\[ L_{\hat x} = \begin{pmatrix} 1 & 0 \\ x & I + L_x \end{pmatrix},\qquad L_{\hat x}^2 = \begin{pmatrix} 1 & 0 \\ 2x+x^2 & I + 2L_x + L_x^2 \end{pmatrix},
\]
$\hat x^2 = (1,2x+x^2)$, and hence
\[ U_{\hat x} = 2L_{\hat x}^2 - L_{(1,2x+x^2)} = \begin{pmatrix} 1 & 0 \\ 2x+x^2 & I + 2L_x + U_x \end{pmatrix}.
\]
By Theorem \ref{th_invertible} the element $\hat x$ is invertible, i.e., $-x$ is quasi-regular, if and only if the matrix $I + 2L_x + U_x$ is invertible. In this case we have
\begin{equation} \label{Uinvunit}
U_{\hat x}^{-1} = \begin{pmatrix} 1 & 0 \\ -(I + 2L_x + U_x)^{-1}(2x+x^2) & (I + 2L_x + U_x)^{-1} \end{pmatrix}.
\end{equation}

Let $\nabla,\hat D$ be the flat affine connections of $J,\hat J$, respectively. For $u \in J$ and $x \in {\cal Y}$, denote by $y$ the directional derivative $\nabla_u(-x)^{(-1)}$. By \eqref{inv_der} and the definition of the quasi-inverse we have $(0,y) = \hat D_{(0,u)}\hat x^{-1} = U_{\hat x}^{-1}(0,u)$. By \eqref{Uinvunit} we then have
\begin{equation} \label{der_q_inv}
\nabla_u(-x)^{(-1)} = (I + 2L_x + U_x)^{-1}u.
\end{equation}
We shall now consider the particular situation when $u = X_w(x)$ for some $w \in J$. Then $L_{\hat x}\hat w = (0,u)$ by definition of the vector field $X_w$. By Theorem \ref{th_invertible} it follows that $(0,y) = U_{(1,x)}^{-1}L_{\hat x}\hat w = L_{\hat x^{-1}}\hat w$. Now note that $\hat x^{-1} = (1,-(-x)^{(-1)})$, and hence $L_{\hat x^{-1}}\hat w = (0,X_w(-(-x)^{(-1)}))$. We therefore obtain
\begin{equation} \label{der_q_inv_special}
\nabla_{X_w(x)}(-x)^{(-1)} = X_w(-(-x)^{(-1)}) = w - w \bullet (-x)^{(-1)}.
\end{equation}

\section{Immersions and Jordan algebras} \label{sec_correspondence}

We are now in a position to establish a connection between non-degenerate graph immersions with parallel cubic form and real metrised Jordan algebras.

\subsection{Jordan algebras defined by immersions} \label{subs_J_i}

In this subsection we show that a non-degenerate graph immersion $f: \Omega \to \mathbb A^{n+1}$ with parallel cubic form equips the tangent space $T_y\Omega$ with the structure of a metrised Jordan algebra $J$ for every point $y \in \Omega$.

Since we consider only a neighbourhood of the point $y$, it is convenient to identify $\Omega$ with a subset of $\mathbb R^n$ and to assume the existence of a representing function $F: \Omega \to \mathbb R$. By Corollary \ref{surface_chars} $f$ has parallel cubic form if and only if $F$ obeys the quasi-linear fourth-order PDE \eqref{quasi_lin_PDE}.

Let us deduce the integrability condition of this PDE. Differentiating \eqref{quasi_lin_PDE} with respect to the coordinate $x^{\eta}$ and substituting the appearing fourth order derivatives of $F$ by the right-hand side of \eqref{quasi_lin_PDE}, we obtain after simplification
\begin{eqnarray*}
F_{,\alpha\beta\gamma\delta\eta} &=& \frac14F^{,\rho\sigma}F^{,\mu\nu}\left(F_{,\beta\eta\nu}F_{,\alpha\rho\mu}F_{,\gamma\delta\sigma} + F_{,\alpha\eta\mu}F_{,\rho\beta\nu}F_{,\gamma\delta\sigma} + F_{,\gamma\eta\nu}F_{,\alpha\rho\mu}F_{,\beta\delta\sigma} + F_{,\alpha\eta\mu}F_{,\rho\gamma\nu}F_{,\beta\delta\sigma} \right. \\
&& + F_{,\beta\eta\nu}F_{,\gamma\rho\mu}F_{,\alpha\delta\sigma} + F_{,\gamma\eta\mu}F_{,\rho\beta\nu}F_{,\alpha\delta\sigma} + F_{,\beta\eta\nu}F_{,\delta\rho\mu}F_{,\alpha\gamma\sigma} + F_{,\delta\eta\mu}F_{,\rho\beta\nu}F_{,\alpha\gamma\sigma} \\
&& \left. + F_{,\delta\eta\nu}F_{,\alpha\rho\mu}F_{,\beta\gamma\sigma} + F_{,\alpha\eta\mu}F_{,\rho\delta\nu}F_{,\beta\gamma\sigma} + F_{,\delta\eta\nu}F_{,\gamma\rho\mu}F_{,\alpha\beta\sigma} + F_{,\gamma\eta\mu}F_{,\rho\delta\nu}F_{,\alpha\beta\sigma}\right).
\end{eqnarray*}
The right-hand side must be symmetric in all 5 indices. Commuting the indices $\delta,\eta$ and equating the resulting expression with the original one we obtain
\begin{eqnarray*}
\lefteqn{F^{,\rho\sigma}F^{,\mu\nu}\left(F_{,\beta\eta\nu}F_{,\delta\rho\mu}F_{,\alpha\gamma\sigma} + F_{,\alpha\eta\mu}F_{,\rho\delta\nu}F_{,\beta\gamma\sigma} + F_{,\gamma\eta\mu}F_{,\rho\delta\nu}F_{,\alpha\beta\sigma} \right.} \\
&& \left. - F_{,\beta\delta\nu}F_{,\eta\rho\mu}F_{,\alpha\gamma\sigma} - F_{,\alpha\delta\mu}F_{,\rho\eta\nu}F_{,\beta\gamma\sigma} - F_{,\gamma\delta\mu}F_{,\rho\eta\nu}F_{,\alpha\beta\sigma}\right) = 0.
\end{eqnarray*}
Raising the index $\eta$, we get by virtue of \eqref{correspondence} the integrability condition
\[ K_{\alpha\mu}^{\eta}K_{\delta\rho}^{\mu}K_{\beta\gamma}^{\rho} + K_{\beta\mu}^{\eta}K_{\delta\rho}^{\mu}K_{\alpha\gamma}^{\rho} + K_{\gamma\mu}^{\eta}K_{\delta\rho}^{\mu}K_{\alpha\beta}^{\rho} = K_{\alpha\delta}^{\mu}K_{\rho\mu}^{\eta}K_{\beta\gamma}^{\rho} + K_{\beta\delta}^{\mu}K_{\rho\mu}^{\eta}K_{\alpha\gamma}^{\rho} + K_{\gamma\delta}^{\mu}K_{\rho\mu}^{\eta}K_{\alpha\beta}^{\rho}.
\]
This condition is satisfied if and only if
\[ K_{\alpha\mu}^{\eta}K_{\delta\rho}^{\mu}K_{\beta\gamma}^{\rho}u^{\alpha}u^{\beta}u^{\gamma}v^{\delta} = K_{\alpha\delta}^{\mu}K_{\rho\mu}^{\eta}K_{\beta\gamma}^{\rho}u^{\alpha}u^{\beta}u^{\gamma}v^{\delta}
\]
for all tangent vector fields $u,v$ on $\Omega$, which can be written as
\begin{equation} \label{int_cond}
K(K(K(u,u),v),u) = K(K(u,v),K(u,u)).
\end{equation}

{\theorem \label{main1} Let $f: \Omega \to \mathbb A^{n+1}$ be a non-degenerate graph immersion with parallel cubic form. Let $y \in \Omega$ be a point and let $\bullet: T_y\Omega \times T_y\Omega \to T_y\Omega$ be the multiplication $(u,v) \mapsto K(u,v)$ defined by the difference tensor $K = \nabla - \hat\nabla$ at $y$. Let $\gamma$ be the symmetric non-degenerate bilinear form defined on $T_y\Omega$ by the affine fundamental form $h$.

Then the tangent space $T_y\Omega$, equipped with the multiplication $\bullet$, is a real Jordan algebra $J$, and $(J,\gamma)$ is a metrised Jordan algebra. }

\begin{proof}
Assume the conditions of the theorem. The tensor $K_{\alpha\beta}^{\gamma}$ is symmetric in the indices $\alpha,\beta$, hence the multiplication $\bullet$ is commutative. With $u^2 = u \bullet u$ condition \eqref{int_cond} can be rewritten as
\[ u \bullet (u^2 \bullet v) = u^2 \bullet (u \bullet v),
\]
which is the Jordan identity in Definition \ref{Jordan_def}. Thus $T_y\Omega$, equipped with the multiplication $\bullet$, is a Jordan algebra $J$.

Let now $u,v,w \in J$ be arbitrary vectors. We obtain
\[ \gamma(u \bullet v,w) = h[K[u,v],w] = -\frac12 C[u,v,w] = \gamma(u,v \bullet w)
\]
by the total symmetry of the cubic form $C$. Hence the form $\gamma$ satisfies \eqref{3symmetry} and is a trace form. Finally, $\gamma$ is non-degenerate because the immersion $f$ is non-degenerate.
\end{proof}


{\lemma \label{AHS} Assume the conditions of Theorem \ref{main1}. If $f$ is an improper affine hypersphere with affine normal equal to $\xi$, then $J$ is nilpotent and $\det\gamma = \pm1$. }

\begin{proof}
Assume the conditions of the lemma. By Lemma \ref{IAHS} we have $\det F'' \equiv \pm1$. Since $\gamma = F''(y)$, we get $\det\gamma = \pm1$. Further, we have
\[ \nabla_u(\log|\det F''|) = F^{,\alpha\beta}F_{,\alpha\beta\gamma}u^{\gamma} = -2K^{\alpha}_{\alpha\gamma}u^{\gamma} = -2\tr L_u = 0
\]
for every tangent vector $u \in T_y\Omega$, because $\det F''$ is constant. By Lemma \ref{niltrace} the Jordan algebra $J$ is a nilalgebra, and by Lemma \ref{nilnil} it is nilpotent.
\end{proof}

{\lemma \label{decomp_lemma} Assume the notations of Definition \ref{def_prod} and let $(J_k,\gamma_k)$ be the metrised Jordan algebras defined by the graph immersions $f_k$, as in Theorem \ref{main1}. Then the metrised Jordan algebra $(J,\gamma)$ defined by the product $f$ of the graph immersions $f_1,\dots,f_r$ is the direct sum of the metrised Jordan algebras $(J_k,\gamma_k)$. }

\begin{proof}
The assertion of the lemma directly follows from the block-diagonal structure of the derivatives $F'',F'''$ which was established in the proof of Lemma \ref{lem_prod}.
\end{proof}

\subsection{Immersions defined by Jordan algebras} \label{sec_imm_alg}

In this subsection we consider the opposite direction and show that every metrised real Jordan algebra $(J,\gamma)$ defines a vertically translationally invariant involutive distribution $\Delta$ on the product ${\cal D} = {\cal Y} \times \mathbb R$, such that the integral hypersurfaces of $\Delta$ are graph immersions with parallel cubic form. Here $-{\cal Y}$ is the connected component of zero in the set of quasi-regular elements of $J$, as defined in Section \ref{sec_Jordan}.

The distribution $\Delta$ shall be defined by the tangent spaces to the orbits of a group acting on ${\cal D}$. For every $w \in J$, define the affine-linear vector field
\[ \tilde X_w(x,\alpha) = (X_w(x),\gamma(w,x))
\]
on $J \times \mathbb R$, where $X_w(x)$ is defined by \eqref{flow_gen} on $J$. Let $\tilde\varphi_w: (J \times \mathbb R) \times \mathbb R \to J \times \mathbb R$ be the flow defined by this vector field on $J \times \mathbb R$. Let $\Theta$ be the group of affine-linear transformations of $J \times \mathbb R$ generated by the flows $\tilde\varphi_w$, $w \in J$. Let further $\pi: J \times \mathbb R \to J$ be the canonical projection, and $D,\nabla$ the flat affine connections on $J \times \mathbb R,J$, respectively.

By definition the flow $\tilde\varphi_w$ on the product $J \times \mathbb R$ projects by virtue of $\pi$ down to the flow $\varphi_w$ generated by \eqref{flow_gen} on $J$, i.e., $\pi \circ \tilde\varphi_w = \varphi_w \circ (\pi \times Id)$. Therefore we may define a surjective group homomorphism ${\cal H}: \Theta \to \Pi_1$ by restricting the action of any element $g \in \Theta$ to the first component of the product $J \times \mathbb R$. In other words, we have $\pi \circ g = {\cal H}(g) \circ \pi$ for every $g \in \Theta$.

Since ${\cal Y}$ is invariant under the action of $\Pi_1$, and the generators of $\Theta$ are invariant under vertical translations, we have that the product set ${\cal D}$ is invariant under the action of $\Theta$.

{\lemma \label{transitive_action4} Let $\tilde a = (a,\alpha) \in {\cal D}$ and $b \in {\cal Y}$ be arbitrary. Then there exists an element $g \in \Theta$ such that $\pi(g\tilde a) = b$. }

\begin{proof}
By Corollary \ref{cor:Pi1transitive} there exists an element $g' \in \Pi_1$ such that $g'a = b$. It then suffices to choose any $g \in {\cal H}^{-1}[g']$.
\end{proof}

Let us define a 1-form $\zeta$ on ${\cal D}$. At the point $(x,\alpha) \in {\cal D}$ we set
\begin{equation} \label{zeta_def}
\zeta(u,\mu) = \gamma(u,(-x)^{(-1)}) - \mu,
\end{equation}
where $(u,\mu) \in J \times \mathbb R$ is a tangent vector at $(x,\alpha)$. Let the distribution $\Delta$ on ${\cal D}$ be defined by the kernel of $\zeta$, and let $\xi$ be the constant vector field on $J \times \mathbb R$ given by $({\bf 0},1)$.

{\lemma The form $\zeta$ does not vanish on ${\cal D}$, is closed, and $\zeta(\xi) \equiv -1$. The distribution $\Delta$ is $n$-dimensional and involutive. }

\begin{proof}
At $(x,\alpha) \in {\cal D}$ we have $\zeta(\xi) = \gamma({\bf 0},(-x)^{(-1)}) - 1 = -1$. Hence $\zeta$ is nowhere zero and its kernel $\Delta$ is an $n$-dimensional distribution.

At $(x,\alpha) \in {\cal D}$ the derivative of $\zeta$ in the direction of a vector $(v,\nu) \in J \times \mathbb R$ is by virtue of \eqref{der_q_inv} given by
\begin{equation} \label{der_zeta}
(D_{(v,\nu)}\zeta)(u,\mu) = \gamma(u,\nabla_v(-x)^{(-1)}) = \gamma(u,(I + 2L_x + U_x)^{-1}v).
\end{equation}
But the operator $(I + 2L_x + U_x)^{-1}$ is self-adjoint with respect to $\gamma$, because the operators $L_x,L_{x^2}$ are self-adjoint by virtue of Definition \ref{def_trace_form}. Thus \eqref{der_zeta} implies $(D_{(v,\nu)}\zeta)(u,\mu) = (D_{(u,\mu)}\zeta)(v,\nu)$. Hence $D\zeta$ is symmetric in its two arguments, and the exterior derivative of $\zeta$ vanishes.
\end{proof}

{\lemma \label{lem:zeta_invariant} The form $\zeta$ and the vector field $\xi$ are invariant under the action of the group $\Theta$ on ${\cal D}$. }

\begin{proof}
For the Lie derivatives of $\zeta$ with respect to the generators of $\Theta$ we have
\begin{eqnarray*}
({\cal L}_{\tilde X_w}\zeta)(u,\mu) &=& (D_{\tilde X_w}\zeta)(u,\mu) + \zeta(D_{(u,\mu)}\tilde X_w) = \gamma(u,\nabla_{X_w}(-x)^{(-1)}) + \zeta(\nabla_uX_w,\gamma(w,u)) \\
&=& \gamma(u,w - w \bullet (-x)^{(-1)}) + \gamma(w \bullet u,(-x)^{(-1)}) - \gamma(w,u) = 0.
\end{eqnarray*}
Here the third equality comes from \eqref{der_q_inv_special}, and the fourth equality holds because $\gamma$ is a trace form. Thus $\zeta$ is invariant with respect to the action of $\Theta$.

The corresponding Lie derivatives of $\xi$ at $(x,\alpha) \in {\cal D}$ are given by
\[ {\cal L}_{\tilde X_w}\xi = -\frac{\partial \tilde X_w(x,\alpha)}{\partial\alpha} = 0.
\]
Here the first equality follows from the fact that $\xi$ is a constant vector field. Hence the vector fields $\xi$ and $\tilde X_w$ commute on ${\cal D}$, and $\xi$ is left invariant by the action of $\Theta$.
\end{proof}

{\corollary The distribution $\Delta$ consists of the tangent spaces to the orbits of the group $\Theta$. }

\begin{proof}
Let $w \in J$. Then at $(x,\alpha) \in {\cal D}$ we have
\[ \zeta(\tilde X_w) = \gamma(w + w \bullet x,(-x)^{(-1)}) - \gamma(w,x) = \gamma(w,(-x)^{(-1)}) + \gamma(w,(-x)^{(-1)} \bullet x) - \gamma(w,x)  = 0,
\]
where the last equality is due to \eqref{quasi_prod}. Hence the vector fields $\tilde X_w$ generating $\Theta$ are tangent to $\Delta$. It follows that the orbit of $\Theta$ through a given point of ${\cal D}$ is contained in the integral manifold of $\Delta$ through that point.

Let $(x,\alpha) \in {\cal D}$ be arbitrary. By Lemma \ref{transitive_action4} there exists $\beta \in \mathbb R$ such that $({\bf 0},\beta)$ is in the orbit of $(x,\alpha)$. At $({\bf 0},\beta)$ we have $\zeta(u,\mu) = -\mu$, and the kernel of $\zeta$ equals the horizontal subspace. On the other hand, $\tilde X_w({\bf 0},\beta) = (w,0)$, and hence every vector of the horizontal subspace is actually in the tangent space to the orbit of $\Theta$ through $({\bf 0},\beta)$. Therefore the claim of the corollary holds at $({\bf 0},\beta)$. Since the group $\Theta$ preserves the form $\zeta$, this relation holds also at the initial point $(x,\alpha)$.
\end{proof}

{\corollary \label{cor:orbits_manifolds} The orbits of the group $\Theta$ in ${\cal D}$ are given by the maximal integral manifolds of the distribution $\Delta$. }

\begin{proof}
The preceding corollary implies that every integral manifold of $\Delta$ must be contained in a single orbit. The opposite inclusion follows from the fact that $\Theta$ is connected.
\end{proof}

%
%

{\lemma \label{lem:h_Dzeta} Let $M \subset {\cal D}$ be an integral manifold of the distribution $\Delta$, equipped with the constant transversal vector field $\xi$. The affine fundamental form $h$ induced on $M$ is given by the restriction of the derivative $D\zeta$ to $M$. }

\begin{proof}
Let $X,Y$ be vector fields on $M$. By definition we have $D_XY = \nabla_XY + h(X,Y)\xi$, where $\nabla$ denotes the induced connection on $M$. Since $\zeta(\nabla_XY) = \zeta(Y) = 0$, $\zeta(\xi) = -1$, we have $h(X,Y) = -\zeta(D_XY) = (D\zeta)(X,Y)$.
\end{proof}

{\lemma \label{lem:nondegenerate} The kernel of the derivative $D\zeta$ is one-dimensional and generated by $\xi$. In particular, $D\zeta$ is non-degenerate on $\Delta$. }

\begin{proof}
By \eqref{der_zeta} the matrix of $D\zeta$ is given by $\diag(\gamma (I + 2L_x + U_x)^{-1},0)$. Both $\gamma$ and $I + 2L_x + U_x$ are non-degenerate. Hence the kernel of $D\zeta$ is one-dimensional and generated by $\xi$. Since $\Delta$ is transversal to $\xi$, the form $D\zeta$ is non-degenerate on $\Delta$.
\end{proof}

{\lemma \label{lem:homogeneous} Let $M$ be a maximal integral manifold of $\Delta$, considered as a graph immersion equipped with the transversal vector field $\xi$. Then $M$ is a homogeneous space on which the group $\Theta$ of affine-linear transformations acts transitively and isometrically. Its image $\pi[M]$ coincides with the domain ${\cal Y}$, the restriction of $\pi$ to $M$ is a covering map, and the different sheets are mapped to each other by vertical translations. }

\begin{proof}
From Corollary \ref{cor:orbits_manifolds} it follows that $\Theta$ acts transitively on $M$. By Lemma \ref{lem:zeta_invariant} the form $\zeta$ and hence its derivatives are invariant with respect to the action of the group $\Theta$. Then the group $\Theta$ preserves the affine fundamental form on $M$ by virtue of Lemma \ref{lem:h_Dzeta}, and $M$ is a homogeneous space. It also follows that $M$ is affine complete.

The relation $\pi[M] = {\cal Y}$ follows from Lemma \ref{transitive_action4}. Since the distribution $\Delta$ is invariant with respect to translations parallel to $\xi$, every simply connected neighbourhood in ${\cal Y}$ is evenly covered by $\pi|_M$. Thus $\pi|_M$ is a covering map, and the different sheets over evenly covered neighbourhoods are related by vertical translations.
\end{proof}

We shall now compute a potential of $\zeta$ in a cylindrical neighbourhood of the line $\{{\bf 0}\} \times \mathbb R \subset {\cal D}$. Let $||\cdot||$ be a Euclidean norm on $J$ such that the matrix of $\gamma$ is orthogonal with respect to this norm. Let $\Upsilon_k$ be the compact set of all possible products of $k$ elements on the unit sphere with respect to the norm $||\cdot||$, and define $\rho_k = \max_{y \in \Upsilon_k} ||y|| \in [0,\infty)$. By induction on $k$ it is not hard to see that $\rho_k \leq \rho_2^{k-1}$. Define $R = \lim\inf_{k \to \infty} \rho_k^{-1/k}$. We have $R \geq \lim_{k \to \infty} \rho_2^{-\frac{k-1}{k}} = \rho_2^{-1} > 0$, and $R \in (0,\infty]$. Let $B_R \subset J$ be the open ball with radius $R$ around ${\bf 0}$.

{\lemma \label{loc_potential} Assume above definitions. Then the function $F(x) = \sum_{k=2}^{\infty} \frac{(-1)^k}{k} \gamma(x,x^{k-1})$ is analytic on $B_R$ and the function $\Phi(x,\alpha) = F(x) - \alpha$ is analytic on the product $B_R \times \mathbb R$. Moreover, $B_R \subset {\cal Y}$, $B_R \times \mathbb R \subset {\cal D}$, and $\Phi$ is a potential of the 1-form $\zeta$ on $B_R \times \mathbb R$. }

\begin{proof}
Let $x \in B_R$ and $||x|| = r < R$. Then $||x^k|| \leq \rho_k r^k$. Since $\gamma$ is orthogonal, it follows that $|\gamma(x,x^{k-1})| \leq r||x^{k-1}|| \leq \rho_{k-1} r^k$. Hence the series defining the function $F(x)$ is uniformly absolutely convergent on every compact set contained in $B_R$. Therefore $F$ is analytic in $B_R$ and $\Phi$ is analytic in $B_R \times \mathbb R$.

The partial derivative of $x^k$ in the direction $u$ is given by
\begin{eqnarray*} \label{grad_xk}
\nabla_u x^k &=& \nabla_u (L_x^{k-1}x) = \sum_{l=1}^{k-1} L_x^{l-1}L_uL_x^{k-1-l}x + L_x^{k-1}u = \sum_{l=1}^{k-1} L_x^{l-1}L_ux^{k-l} + L_x^{k-1}u \\ &=& \sum_{l=1}^k L_x^{l-1}L_{x^{k-l}}u,
\end{eqnarray*}
where $L_{x^0}$ is by convention the identity matrix. The derivative of $F$ is then given by
\begin{eqnarray} \label{derFu}
\nabla_u F &=& \sum_{k=2}^{\infty} \frac{(-1)^k}{k} \left( \gamma(\nabla_u x,x^{k-1}) + \gamma(x,\nabla_u x^{k-1}) \right) \nonumber\\ &=& \sum_{k=2}^{\infty} \frac{(-1)^k}{k} \left( \gamma(u,x^{k-1}) + \sum_{l=1}^{k-1} \gamma(x,L_x^{l-1}L_{x^{k-1-l}}u) \right) \nonumber\\ &=& \sum_{k=2}^{\infty} \frac{(-1)^k}{k} \left( \gamma(x^{k-1},u) + \sum_{l=1}^{k-1} \gamma(L_{x^{k-1-l}}L_x^{l-1}x,u) \right) = \sum_{k=2}^{\infty} (-1)^k \gamma(x^{k-1},u) \nonumber\\ &=& \sum_{k=1}^{\infty} (-1)^{k+1} \gamma(x^k,u),
\end{eqnarray}
where the fourth equality comes from power-associativity and all sums define analytic functions in $B_R$.

For every $x \in B_R$ we have
\[ (1,x) \hat\bullet (1,\sum_{k=1}^{\infty}(-x)^k) = (1,x+\sum_{k=1}^{\infty}(-x)^k-\sum_{k=1}^{\infty}(-x)^{k+1}) = (1,{\bf 0}) = \hat e.
\]
From Lemmas \ref{strongly_ass} and \ref{quasi_char} it then follows that
\begin{equation} \label{quasi_inv_analytic}
(-x)^{(-1)} = -\sum_{k=1}^{\infty}(-x)^k.
\end{equation}
Thus every element in $B_R$ is quasi-invertible and $B_R \subset {\cal Y}$, $B_R \times \mathbb R \subset {\cal D}$.

By \eqref{quasi_inv_analytic} the form $\zeta$ defined by \eqref{zeta_def} is given by $\zeta(u,\mu) = \sum_{k=1}^{\infty}(-1)^{k+1}\gamma(u,x^k) - \mu$. From \eqref{derFu} it then follows that $\zeta$ has the potential $\Phi(x,\alpha) = F(x) - \alpha$.
\end{proof}

{\lemma \label{lem:analytic_inverse} Assume the notations of Lemma \ref{loc_potential}, and set $\tilde x = -(-x)^{(-1)}$. Then the identity $F(x) + F(\tilde x) + \gamma(x,\tilde x) \equiv 0$ holds on $B_R$. }

\begin{proof}
Since $\gamma$ is a trace form, we have $\gamma(x^k,x^l) = \gamma(x,x^{k+l-1})$ for all $k,l \geq 1$ and hence formally
\begin{align*}
F(\tilde x) &= \sum_{k=2}^{\infty} \frac{(-1)^k}{k} \sum_{l = k}^{\infty} (-1)^l \begin{pmatrix} l-1 \\ k-1 \end{pmatrix} \gamma(x,x^{l-1}) = \sum_{l=2}^{\infty} \frac{(-1)^l}{l} \gamma(x,x^{l-1}) \sum_{k=2}^l (-1)^k\begin{pmatrix} l \\ k \end{pmatrix} \\ &= \sum_{l=2}^{\infty} \frac{(-1)^l(l-1)}{l} \gamma(x,x^{l-1}).
\end{align*}
By virtue of $\sum_{k=0}^l \begin{pmatrix} l \\ k \end{pmatrix} = 2^l$ the power series converge absolutely in $B_{R/2}$. In this smaller ball we then get
\[ F(\tilde x) + F(x) = \sum_{l=2}^{\infty} (-1)^l \gamma(x,x^{l-1}) = -\gamma(x,\tilde x).
\]
By analytic continuation this identity then also holds on $B_R$.
\end{proof}

{\corollary \label{local_rep} Assume the notations of Lemma \ref{loc_potential} and let $M$ be the maximal integral manifold through $({\bf 0},\beta) \in {\cal D}$ of the involutive distribution $\Delta$. Then $M$ contains the graph $M'$ of the function $F(x)+\beta$. }

\begin{proof}
For $x \in B_R$ we have $\Phi(x,F(x)+\beta) = -\beta$, and the graph of $F(x)+\beta$ is a level surface of the function $\Phi$. Hence by Lemma \ref{loc_potential} this graph is an integral manifold of the distribution $\Delta$. Note that $M'$ is connected. Since $F({\bf 0}) = 0$, the point $({\bf 0},\beta) \in M$ is contained in $M'$. Hence $M' \subset M$.
\end{proof}

{\theorem \label{alg_to_imm} Let $M$ be a maximal integral manifold of the distribution $\Delta$, equipped with the transversal vector field $\xi$. Then $M$ is a homogeneous, symmetric, analytic, non-degenerate, affine complete graph immersion with parallel cubic form. The restriction $\pi|_M$ is a covering map of the domain ${\cal Y} \subset J$, and different sheets of the covering are related by a translation parallel to the vector $\xi$. }

\begin{proof}
By virtue of Lemmas \ref{lem:nondegenerate}, \ref{lem:homogeneous} it remains to show that $M$ is symmetric, analytic, and has parallel cubic form.

That $M$ is analytic follows from the analyticity of $\zeta$.

Let $y = ({\bf 0},\beta) \in M$ be an arbitrary point in the preimage of ${\bf 0}$ with respect to the covering $\pi|_M$. We shall construct a symmetry $\varsigma$ of $M$ around $y$.

To the point $(x,\alpha) \in M$ we associate the point $\varsigma(x,\alpha) = (-(-x)^{(-1)},2\beta-\alpha+\gamma(x,(-x)^{(-1)}))$. Direct verification yields that $y$ is a fixed point of $\varsigma$, and that $\varsigma$ is an inversion. By Lemma \ref{lem:analytic_inverse} and Corollary \ref{local_rep} the map $\varsigma$ takes a neighbourhood of $y$ in $M$ to $M$. By analytic continuation it then maps $M$ to $M$. By \eqref{quasi_inv_analytic} the differential of $\varsigma$ at $y$ is an inversion of the tangent space $T_yM$, multiplying every tangent vector by $-1$.

We shall now show that $\varsigma$ is an isometry. Let us introduce locally the coordinate $x \in {\cal Y}$ on $M$. By \eqref{der_zeta} and Lemma \ref{lem:h_Dzeta} the affine fundamental form $h$ at $x$ is given by $h(u,v) = \gamma(u,(I + 2L_x + U_x)^{-1}v)$, where $u,v \in T_xM \cong J$ are tangent vectors at $x$. By \eqref{der_q_inv} the images $\tilde u,\tilde v \in T_{\varsigma(x)}M \cong J$ of $u,v$ under the differential of the map $\varsigma$ are given by $\tilde u = -(I + 2L_x + U_x)^{-1}u$, $\tilde v = -(I + 2L_x + U_x)^{-1}v$. We then have
\begin{eqnarray*}
h(\tilde u,\tilde v) &=& \gamma(\tilde u,(I + 2L_{\varsigma(x)} + U_{\varsigma(x)})^{-1}\tilde v) \\ &=& \gamma((I + 2L_x + U_x)^{-1}u,(I + 2L_{\varsigma(x)} + U_{\varsigma(x)})^{-1}(I + 2L_x + U_x)^{-1}v) \\ &=& \gamma(u,(I + 2L_x + U_x)^{-1}(I + 2L_{\varsigma(x)} + U_{\varsigma(x)})^{-1}(I + 2L_x + U_x)^{-1}v).
\end{eqnarray*}
Here in the third equality we used that $I + 2L_x + U_x$ is self-adjoint with respect to $\gamma$. Now note that $\varsigma$ is an inversion, i.e., $\varsigma^2$ is the identity map. Therefore the differential of $\varsigma$ at $x$ is the inverse of the differential of $\varsigma$ at $\varsigma(x)$, and hence $(I + 2L_x + U_x)^{-1}(I + 2L_{\varsigma(x)} + U_{\varsigma(x)})^{-1} = I$. It follows that
\[ h(\tilde u,\tilde v) = \gamma(u,(I + 2L_x + U_x)^{-1}v) = h(u,v),
\]
which proves our claim.

We have constructed an isometric inversion of $M$ around the point $y$. Since $M$ is a homogeneous space, this inversion generates an isometric inversion around an arbitrary point $y' \in M$ by conjugation with an isometry taking $y$ to $y'$. Thus $M$ is a symmetric space.

Let us show that $M$ has parallel cubic form. Choose again $y = ({\bf 0},\beta) \in M$. By Corollary \ref{local_rep} $M$ can locally around $y$ be represented as the graph of the function $F(x) + \beta$. Note that the powers $x^k$ are homogeneous polynomials of degree $k$ in the entries of $x$. By \eqref{derFu} we then have at $x = {\bf 0}$
\begin{eqnarray} \label{derF0}
F_{,\alpha\beta}u^{\alpha}v^{\beta} &=& \gamma(v,u), \nonumber\\
F_{,\alpha\beta\gamma}u^{\alpha}u^{\beta}v^{\gamma} &=& -2\gamma(u \bullet v,u) = -2\gamma(u^2,v), \nonumber\\
F_{,\alpha\beta\gamma\delta}u^{\alpha}u^{\beta}u^{\gamma}u^{\delta} &=& 6\gamma(u^3,u) = 6\gamma(u^2,u^2).
\end{eqnarray}
The matrix of the Hessian $F''$ is then equal to the matrix of $\gamma$, and the coordinate vector of the 1-form $F_{,\alpha\beta\gamma}u^{\alpha}u^{\beta}$ is given by $-2\gamma u^2$. Hence
\[ \frac32F^{,\rho\sigma}F_{,\alpha\beta\rho}F_{,\gamma\delta\sigma}u^{\alpha}u^{\beta}u^{\gamma}u^{\delta} = \frac32(-2\gamma u^2)^T\gamma^{-1}(-2\gamma u^2) = 6(u^2)^T\gamma u^2 = 6\gamma(u^2,u^2),
\]
and \eqref{quasi_lin_PDE} is satisfied at $x = 0$. By \eqref{parallel_F3} we then have $\hat\nabla C = 0$ at $y$.

However, the action of the symmetry group $\Theta$ preserves the covariant derivative $\hat\nabla C$ of the cubic form on $M$, because $\Theta$ consists of affine-linear transformations of the ambient space $J \times \mathbb R \sim \mathbb A^{n+1}$ and preserves the transversal vector field $\xi$ by Lemma \ref{lem:zeta_invariant}. Since $\Theta$ acts transitively on $M$, we have $\hat\nabla C = 0$ identically on $M$. The manifold $M$ has then parallel cubic form. This completes the proof.
\end{proof}

Let us now consider the case when the Jordan algebra $J$ is nilpotent.

{\theorem \label{S1} Let $(J,\gamma)$ be a real metrised nilpotent Jordan algebra such that $|\det\gamma| = 1$. Assume the notations at the beginning of this subsection, and let $M$ be the maximal integral manifold through $({\bf 0},0) \in {\cal D}$ of the distribution $\Delta$, equipped with the transversal vector field $\xi$. Then $M$ is a graph immersion with the representing function given by the polynomial $F(x) = \sum_{k=2}^{\infty} \frac{(-1)^k}{k} \gamma(x,x^{k-1})$. It is a non-degenerate, affine complete, Euclidean complete, homogeneous, symmetric improper affine hypersphere with parallel cubic form. }

\begin{proof}
First note that the sum defining the function $F$ is finite, since $J$ is nilpotent. Hence $F$ is a polynomial. Moreover, the ball $B_R$ from Lemma \ref{loc_potential} is equal to the whole space $J$, because there exists a $k_0 \in \mathbb N$ such that $\rho_k = 0$ for all $k \geq k_0$. By Lemma \ref{nilY} we also have ${\cal Y} = J$. By Corollary \ref{local_rep} the manifold $M$ contains the graph of the function $F$. Clearly the graph of $F$ is Euclidean complete, and hence $M$ equals this graph.

That $M$ is affine complete, non-degenerate, homogeneous, symmetric, and has parallel cubic form, follows from Theorem \ref{alg_to_imm}.

We shall now show that $M$ is an improper affine hypersphere. Differentiating \eqref{derFu} in the direction $v$, we obtain again by \eqref{grad_xk}
\[ \nabla_v\nabla_u F = \sum_{k=1}^{\infty} (-1)^{k+1} \gamma(\nabla_v x^k,u) = \sum_{k=1}^{\infty} (-1)^{k+1} \sum_{l=1}^k \gamma(L_x^{l-1}L_{x^{k-l}}v,u).
\]
Hence the Hessian of $F$ is given by the matrix product
\[ F'' = \gamma \cdot \left( \sum_{k=1}^{\infty} (-1)^{k+1} \sum_{l=1}^k L_x^{l-1}L_{x^{k-l}} \right) = \gamma \cdot \left( \sum_{k=0}^{\infty} \sum_{l=0}^{\infty} (-1)^{k+l} L_x^l L_{x^k} \right).
\]
Now any polynomial with zero constant term in the matrices $L_{x^k}$, $k \geq 1$, is nilpotent. Hence the matrix in parentheses is the sum of the identity matrix and a nilpotent matrix. Therefore its determinant equals 1, and $\det F'' = \det\gamma$ identically on $J$. The assumption $|\det\gamma| = 1$ together with Lemma \ref{IAHS} conclude the proof.
\end{proof}

Finally, we shall consider direct sums of metrised Jordan algebras.

{\lemma \label{direct_sums} Let $(J_k,\gamma_k)$, $k = 1,\dots,r$, be metrised Jordan algebras, and let $(J,\gamma)$ be their direct sum. Let ${\cal Y}_k \subset J_k$, ${\cal Y} \subset J$ be defined as in Subsection \ref{sec_inv}, and let ${\cal D}_k = {\cal Y}_k \times \mathbb R$, ${\cal D} = {\cal Y} \times \mathbb R$. Let $\Delta_k,\Delta$ be the involutive distributions defined on ${\cal D}_k,{\cal D}$, respectively, by the kernels of the corresponding forms \eqref{zeta_def}. Let $M_k,M$ be the maximal integral manifolds through $({\bf 0},0)$ of $\Delta_k,\Delta$, respectively. Then $M$ is locally around $({\bf 0},0)$ a product of $M_k$ as defined in Definition \ref{def_prod}. If the affine hypersurface immersions $M_k$ are improper affine hyperspheres, then $M$ is globally the product of the $M_k$. }

\begin{proof}
The powers of an element $x = (x_1,\dots,x_r) \in J$ are given by $x^k = (x_1^k,\dots,x_r^k)$. We thus have $\gamma(x,x^{l-1}) = \sum_{k=1}^r \gamma_k(x_k,x_k^{l-1})$. By Corollary \ref{local_rep} the hypersurfaces $M_k$ can locally around $({\bf 0},0)$ be represented as graphs of the functions $F_k(x_k) = \sum_{l=2}^{\infty} \frac{(-1)^l}{l}\gamma_k(x_k,x_k^{l-1})$. The product of these graphs as defined in Definition \ref{def_prod} is the graph of the function
\[ F(x) = \sum_{k=1}^r F_k(x^k) = \sum_{k=1}^r \sum_{l=2}^{\infty} \frac{(-1)^l}{l}\gamma_k(x_k,x_k^{l-1}) = \sum_{l=2}^{\infty} \frac{(-1)^l}{l}\gamma(x,x^{l-1}).
\]
Again by Corollary \ref{local_rep} this graph then coincides with $M$ in a neighbourhood of $({\bf 0},0)$.

Finally, if the $M_k$ are improper affine hyperspheres, then by Theorem \ref{S1} $M_k$ is globally the graph of the polynomial $F_k$ and $J_k$ is nilpotent. Therefore $J$ is also nilpotent and $M$ is globally the graph of the polynomial $F(x) = \sum_{k=1}^r F_k(x^k)$.
\end{proof}

\subsection{Immersions and Jordan algebras} \label{sec4_3}

Let us summarize the results of the two preceding subsections.

Let $f: \Omega \to \mathbb A^{n+1}$ be a non-degenerate graph immersion with parallel cubic form. Let $K = \nabla - \hat\nabla$ be the difference tensor between the induced affine connection $\nabla$ on $\Omega$ and the Levi-Civita connection $\hat\nabla$ of the Hessian pseudo-metric defined by the affine fundamental form $h$. Choose a point $y \in \Omega$ and consider the tensor $K$ at $y$. This tensor defines a real Jordan algebra $J_y$ on $T_y\Omega$. Let further $\gamma_y$ be the symmetric bilinear form defined on $T_y\Omega$ by $h$. Then $\gamma_y$ is a non-degenerate trace form on $J_y$ and the pair $(J_y,\gamma_y)$ is a metrised Jordan algebra.

On the other hand, let $J$ be a real Jordan algebra of dimension $n$, $\gamma$ a non-degenerate trace form on $J$, ${\cal Y} \subset J$ the connected component of ${\bf 0}$ in the set of the negatives of the quasi-invertible elements of $J$, and $\Delta$ the involutive distribution on ${\cal D} = {\cal Y} \times \mathbb R$ defined by the kernel of the form \eqref{zeta_def}. Then every maximal integral manifold of $\Delta$ is a non-degenerate graph immersion with parallel cubic form.

We shall now consider the interplay between these relations.

{\lemma \label{imm_alg_imm} Assume above notations and suppose that $\Omega$ is connected. Set $J = J_y$, $\gamma = \gamma_y$. Under the identification of $T_y\Omega \times \mathbb R = J \times \mathbb R$ with $\mathbb A^{n+1}$ given by \eqref{identification} the hypersurface $M = f[\Omega]$ is an integral manifold of $\Delta$. }

\begin{proof}
By virtue of \eqref{identification} we shall assume that the target space of the immersion $f$ is the product $J \times \mathbb R$. Then $M$ is locally around $f(y)$ given by the graph of a function $\tilde F$, and the transversal vector field $\xi$ is given by the constant vector $({\bf 0},1)$. By construction $M$ contains the point $f(y) = ({\bf 0},0)$, and the hyperplane $J \times \{0\}$ is tangent to $M$ at this point. Therefore $\tilde F({\bf 0}) = 0$, $\tilde F'({\bf 0}) = 0$.

Let $M'$ be the maximal integral manifold of $\Delta$ passing through the point $({\bf 0},0)$. By Corollary \ref{local_rep} the hypersurface $M'$ is given by the graph of the function $F(x) = \sum_{k=2}^{\infty} \frac{(-1)^k}{k}\gamma(x,x^{k-1})$ in a neighbourhood of $({\bf 0},0)$. Note that $M'$ has parallel cubic form by Theorem \ref{alg_to_imm}. We have $F({\bf 0}) = 0$, $F'({\bf 0}) = 0$, and the higher derivatives of $F$ at ${\bf 0}$ are given by \eqref{derF0}.

On the other hand, we have $\tilde F''({\bf 0}) = \gamma_y$ by definition of $\gamma_y$, and
\[ \tilde F_{,\alpha\beta\gamma}u^{\alpha}u^{\beta}u^{\gamma} = -2K^{\delta}_{\alpha\beta}\tilde F_{,\gamma\delta}u^{\alpha}u^{\beta}u^{\gamma} = -2\gamma_y(u^2,u)
\]
for all $u \in J_y$. Here the first equality comes form \eqref{correspondence}, and the second one comes from the relation $\tilde F''({\bf 0}) = \gamma_y$.

Thus at ${\bf 0}$ the values of $F$ and its first three derivatives coincide with the values of $\tilde F$ and its corresponding derivatives, respectively. It follows that the hypersurfaces $M,M'$ make a contact of order 3 at ${\bf 0}$. But both hypersurfaces are graph immersions with parallel cubic form and must hence locally coincide. Since $M'$ is affine complete by Theorem \ref{alg_to_imm}, $M$ is actually contained in $M'$. The claim of the lemma now easily follows.
\end{proof}

{\theorem \label{th_main3} Let $f: \Omega \to \mathbb A^{n+1}$ be a connected non-degenerate graph immersion with parallel cubic form and with constant transversal vector field $\xi$. Then $f[\Omega]$ can be extended to an embedded hypersurface $M \subset \mathbb A^{n+1}$ such that if $M$ is equipped with the transversal vector field $\xi$, it is a graph immersion with parallel cubic form satisfying the properties listed in Theorem \ref{alg_to_imm}. } 

\begin{proof}
The theorem is a consequence of Lemma \ref{imm_alg_imm} and Theorem \ref{alg_to_imm}.
%
\end{proof}

Theorem \ref{th_main3} completely characterizes graph immersions with parallel cubic form. It reduces their study to the study of real metrised Jordan algebras $(J,\gamma)$. Note that the choice of different points $y \in \Omega$ leads to metrised Jordan algebras $(J_y,\gamma_y)$ which are isomorphic. The isomorphism is given by the affine-linear maps $g \in \Theta$ which take the different points of $f[\Omega]$ to each other. Thus the isomorphism classes of graph immersions with parallel cubic form are in bijective correspondence with the isomorphism classes of metrised Jordan algebras.

{\remark \label{asso_flat} A stronger condition than the Jordan identity $[L_x,L_{x^2}] = 0$ is the associativity of the algebra $J$, which is equivalent to the condition $[L_x,L_y] = 0$ for all $x,y \in J$. By \cite[eq.~(1.7)]{Duistermaat01} this condition is equivalent to the flatness of the affine metric of the corresponding graph immersions with parallel cubic form. }

We have the following result.

{\theorem \label{th_AHS2} Let $f: \Omega \to \mathbb A^{n+1}$ be an improper affine hypersphere with parallel cubic form. Then $f$ can be extended to an affine complete, Euclidean complete, homogeneous, symmetric improper affine hypersphere with parallel cubic form, which can be represented as the graph of the polynomial $F(x) = \sum_{k=2}^{\infty} \frac{(-1)^k}{k} \gamma(x,x^{k-1})$ for some nilpotent metrised Jordan algebra $(J,\gamma)$ with $|\det\gamma| = 1$. This improper affine hypersphere can be represented as a product of lower-dimensional improper affine hyperspheres with parallel cubic form, in the sense of Definition \ref{def_prod}, if and only if $(J,\gamma)$ is the direct sum of lower-dimensional metrised nilpotent Jordan algebras whose trace forms have determinant equal to $\pm1$. }

\begin{proof}
First of all, note that Theorem \ref{th_main3} is applicable. By Lemma \ref{AHS} the Jordan algebra $J$ is nilpotent and $|\det\gamma| = 1$. The claims of the theorem about the extension of $f$ follow from Theorem \ref{S1}. The assertion about decomposability follows from Lemmas \ref{decomp_lemma} and \ref{direct_sums}.
\end{proof}

Thus the study of improper affine hyperspheres with parallel cubic form can be reduced to the study of those real nilpotent metrised Jordan algebras which cannot be decomposed into a direct sum of lower-dimensional such algebras. We shall call such algebras \emph{indecomposable}. An algebra which can be decomposed shall be called \emph{decomposable}.

\section{Improper affine hyperspheres} \label{sec_classification}

In this section we provide some results on the structure of improper affine hyperspheres with parallel cubic form and furnish a classification of irreducible such hyperspheres in dimension up to 5. In Subsection \ref{subs_Cayley} we consider the special case of the Cayley hypersurfaces.

\subsection{Structural results}

In the preceding section we have proven that improper affine hyperspheres with parallel cubic form are in one-to-one correspondence with nilpotent metrised Jordan algebras whose metrising form has determinant $\pm1$. In this subsection we derive and investigate a semi-canonical form of such algebras. This will subsequently allow us a classification up to isomorphism in low dimensions.

For a real commutative algebra $A$ with multiplication $\bullet$, define the {\it central ascending series} as a sequence of ideals $\{{\bf 0}\} = C_0 \subset C_1 \subset \dots \subset C_k \dots$ by $C_{k+1} = \{ x \in A \,|\, x \bullet y \in C_k\ \forall\ y \in A \}$. If $A$ is nilpotent and finite-dimensional, then there exists an index $m$ such that $C_m = A$ \cite[p.~193]{Bordemann97}. Clearly in this case the series is strictly increasing up to the index $m$, and hence $\{{\bf 0}\} = C_0 \subset \dots \subset C_m = A$ is a partial flag of linear subspaces.

{\lemma \label{complete_flag} Let $A$ be a real commutative nilpotent algebra of dimension $n$ and $\{{\bf 0}\} = C_0 \subset \dots \subset C_m = A$ its central ascending series. Let $\{{\bf 0}\} = V_0 \subset \dots \subset V_n = A$  be a completion of the partial flag $\{C_k\}$ to a complete flag. Then the subspaces $V_k$ are ideals of $A$, and $x \bullet y \in V_{k-1}$ for every $x \in V_k$, $1 \leq k \leq n$, $y \in A$. }

\begin{proof}
Let $1 \leq k \leq n$. Then there exists $k'$ such that $C_{k'} \subset V_{k-1}$, $V_k \subset C_{k'+1}$. For every $x \in V_k \subset C_{k'+1}$, $y \in A$ we have by definition $x \bullet y \in C_{k'} \subset V_{k-1}$. Since $V_{k-1} \subset V_k$, it also follows that $V_k$ is an ideal.
\end{proof}

Let $\gamma$ be a non-degenerate symmetric bilinear form defined on a vector space $V$. For any subspace $U \subset V$, let $U^{\perp} = \{ x \in V \,|\, \gamma(x,y) = 0\ \forall\ y \in U \}$ be its orthogonal subspace with respect to $\gamma$. Since $\gamma$ is non-degenerate, we have $\dim U + \dim U^{\perp} = \dim V$. However, $V = U \oplus U^{\perp}$ if and only if the restriction of $\gamma$ to $U$ is non-degenerate. We now derive a canonical form for $\gamma$ if $V$ is equipped with a complete flag of subspaces.

{\lemma \label{can_form} Let $V$ be a real vector space of dimension $n$ equipped with a complete flag $\{{\bf 0}\} = V_0 \subset \dots \subset V_n = V$ of subspaces. Let $\gamma$ be a non-degenerate symmetric bilinear form on $V$. Then there exists an ordered basis $(v_1,\dots,v_n)$ of $V$ and a partition $S$ of the index set $\{1,\dots,n\}$ into subsets of cardinality 1 or 2 with the following properties.

i) $V_k = \spa\{v_1,\dots,v_k\}$ for all $1 \leq k \leq n$,

ii) $\gamma(v_i,v_j) = 0$ for every $i,j$ belonging to different subsets $I,J \in S$

iii) $\gamma(v_i,v_i) = \pm1$ for every $i$ such that $\{i\} \in S$

iv) $\gamma(v_i,v_i) = \gamma(v_j,v_j) = 0$ and $\gamma(v_i,v_j) = 1$ for every $i \not= j$ such that $\{i,j\} \in S$.
}

\begin{proof}
We prove the lemma by induction over the dimension $n$. For $n = 0$ there is nothing to prove.

Let $n > 0$. We shall consider two cases.

1. The restriction of $\gamma$ to $V_1$ is not zero. Then there exists $v_1 \in V_1$ such that $\gamma(v_1,v_1) = \pm1$, $V = V_1 \oplus V_1^{\perp}$, and the restriction $\gamma'$ of the form $\gamma$ to $V_1^{\perp}$ is non-degenerate. Define subspaces $V'_k \subset V_1^{\perp}$, $k = 0,\dots,n-1$, by $V'_k = V_{k+1} \cap V_1^{\perp}$. Since $V_1 \subset V_{k+1}$ for all such $k$, we have $\dim V'_k = \dim V_{k+1} - 1 = k$, and $V_{k+1} = V'_k \oplus V_1$. It follows that $\{{\bf 0}\} = V'_0 \subset \dots \subset V'_{n-1} = V_1^{\perp}$ is a complete flag of subspaces of $V_1^{\perp}$. Let $\{v_1',\dots,v_{n-1}'\}$ be the ordered basis of $V_1^{\perp}$ and $S'$ the partition of $\{1,\dots,n-1\}$ satisfying properties i) --- iv) by the induction hypothesis on $V_1^{\perp}$ and $\gamma'$. Define $v_k = v_{k-1}'$, $k = 2,\dots,n$, let $\tilde S'$ be the partition of $\{2,\dots,n\}$ obtained from $S'$ by adding 1 to each index, and set $S = \{1\} \cup \tilde S'$. Then the basis $\{v_1,\dots,v_n\}$ of $V$ and the partition $S$ of $\{1,\dots,n\}$ are easily seen to satisfy i) --- iv).

2. The restriction of $\gamma$ to $V_1$ is zero. Since $\gamma$ is non-degenerate, there exists an index $l \geq 2$ such that $V_{l-1} \subset V_1^{\perp}$, $V_l \not\subset V_1^{\perp}$. Choose an arbitrary nonzero vector $v_1 \in V_1$. Then the set $A = \{ x \in V_l \,|\, \gamma(v_1,x) = 1 \}$ is a proper affine hyperplane in $V_l$ which is parallel to $V_{l-1}$. Choose an arbitrary vector $\tilde v \in A$ and define $v_l = \tilde v - \frac{\gamma(\tilde v,\tilde v)}{2}v_1 \in A$. We then have $\gamma(v_l,v_l) = \gamma(\tilde v,\tilde v) - \gamma(\tilde v,\tilde v)\gamma(\tilde v,v_1) = 0$, and $\gamma(v_1,v_l) = 1$, because $v_l \in A$. Since the restriction of $\gamma$ to the 2-dimensional subspace $\spa\{v_1,v_l\}$ is non-degenerate, its restriction $\gamma'$ to the subspace $V' = (\spa\{v_1,v_l\})^{\perp} = V_1^{\perp} \cap (\spa\{v_l\})^{\perp}$ is non-degenerate too, and we have $V = V_1 \oplus \spa\{v_l\} \oplus V'$. For $k = 0,\dots,l-2$ define subspaces $V'_k = V_{k+1} \cap V'$, for $k = l-1,\dots,n-2$ define subspaces $V'_k = V_{k+2} \cap V'$. We have $v_1 \in V_k \cap V_1^{\perp}$, and hence $V_k \cap V_1^{\perp} \not\subset (\spa\{v_l\})^{\perp}$ for all $k \geq 1$. It follows that $\dim(V_k \cap V') = \dim(V_k \cap V_1^{\perp}) - 1$ for all $k \geq 1$. But $\dim(V_k \cap V_1^{\perp}) = \dim V_k = k$ for $k < l$ and $\dim(V_k \cap V_1^{\perp}) = \dim V_k - 1 = k-1$ for $k \geq l$. Thus $\dim V'_k = k$ for all $k$, and $\{{\bf 0}\} = V'_0 \subset \dots \subset V'_{n-2} = V'$ is a complete flag of subspaces of $V'$. Let $\{v_1',\dots,v_{n-2}'\}$ be the ordered basis of $V'$ and $S'$ the partition of $\{1,\dots,n-2\}$ satisfying properties i) --- iv) by the induction hypothesis on $V'$ and $\gamma'$. Define $v_k = v_{k-1}'$ for $k = 2,\dots,l-1$, $v_k = v'_{k-2}$ for $k \geq l+1$, let $\tilde S'$ be the partition of $\{2,\dots,l-1,l+1,\dots,n\}$ obtained from $S'$ by adding 1 to each index $i \leq l-2$ and 2 to each index $i \geq l-1$, and set $S = \{1,l\} \cup \tilde S'$. Then the basis $\{v_1,\dots,v_n\}$ of $V$ and the partition $S$ of $\{1,\dots,n\}$ are easily seen to satisfy i) --- iv).
\end{proof}

{\definition \label{can_J} Let $(J,\gamma)$ be a real nilpotent metrised Jordan algebra of dimension $n$. We shall say that $(J,\gamma)$ is in {\sl semi-canonical form} if the following holds:

i) $x \bullet y \in \spa\{e_1,\dots,e_{k-1}\}$ for every $x \in \spa\{e_1,\dots,e_k\}$, $1 \leq k \leq n$, and $y \in J$. Equivalently, the structure coefficients of $J$ satisfy $K^{\delta}_{\alpha\beta} = 0$ whenever $\delta \geq \min(\alpha,\beta)$.

ii) There exists a partition $S$ of $\{1,\dots,n\}$ into subsets of cardinality 1 or 2 such that the matrix of $\gamma$ is block-diagonal with respect to this partition with diagonal blocks equal to $+1,-1$, or $\begin{pmatrix} 0 & 1 \\ 1 & 0 \end{pmatrix}$.

Here $e_1,\dots,e_n$ are the standard basis vectors of $J$. }

The next result concerns the interplay between the semi-canonical form and the property of irreducibility.

{\lemma \label{semi_can_red} Let $(J,\gamma)$ be a real nilpotent metrised Jordan algebra of dimension $n$ in semi-canonical form. If there exists $1 \leq k < n$ such that the subspaces $V = \spa\{e_1,\dots,e_k\}$ and $V' = \spa\{e_{k+1},\dots,e_n\}$ are mutually $\gamma$-orthogonal, then $(J,\gamma)$ is decomposable. }

\begin{proof}
We have to show that both subspaces $V,V'$ are ideals. That $V$ is an ideal follows from property i) in Definition \ref{can_J}. Let $u \in J$, $v \in V$, $w \in V'$ be arbitrary. We then have $\gamma(u \bullet w,v) = \gamma(w,u \bullet v) = 0$. Here the first equality holds by \eqref{3symmetry} and the second one because $u \bullet v \in V$ and $w \in V'$. Hence $u \bullet w$ is $\gamma$-orthogonal to the whole subspace $V$ and must be an element of $V'$. Thus $V'$ is also an ideal. This completes the proof.
\end{proof}

{\theorem \label{th_can} Every real nilpotent metrised Jordan algebra $(J,\gamma)$ with $\det\gamma = \pm1$ can be brought to semi-canonical form in the sense of Definition \ref{can_J} by an appropriate coordinate transformation. In addition, if $m$ is the number of blocks of size 2 in this form, then the partition $S$ can be chosen equal to $\{\{1,n\},\{2,n-1\},\dots,\{m,n+1-m\},\{m+1\},\dots,\{n-m\}\}$, and $\gamma$ may be chosen such that $\gamma(e_k,e_k) \geq \gamma(e_l,e_l)$ for all $m+1 \leq k < l \leq n-m$. }

\begin{proof}
Let $\{{\bf 0}\} = V_0 \subset \dots \subset V_n = J$ be a complete flag of subspaces obtained by completing the central ascending series of $J$. Let $(v_1,\dots,v_n)$ be an ordered basis of $J$ and $S$ a partition of $\{1,\dots,n\}$ satisfying i) --- iv) in Lemma \ref{can_form}. Choose a coordinate system on $J$ by setting $e_k = v_k$ for all $k$. Then by properties ii) --- iv) of Lemma \ref{can_form} the matrix of $\gamma$ is as required in ii) of Definition \ref{can_J}. By property i) of Lemma \ref{can_form} and by Lemma \ref{complete_flag} property i) of Definition \ref{can_J} holds too, and $(J,\gamma)$ is in semi-canonical form.

Let us show that the partition $S$ can be brought to the required form by a sequence of permutations of coordinates. First we show that for $1 \leq i \leq n-1$ we can exchange the basis vectors $e_i,e_{i+1}$ without invalidating property i) of Definition \ref{can_J} if the partition $S$ satisfies one of the following conditions:

1) $S$ contains subsets $\{i\},\{i+1,l\}$ with $i+1 < l$;

2) $S$ contains subsets $\{k,i\},\{i+1\}$ with $k < i$;

3) $S$ contains subsets $\{k,i\},\{i+1,l\}$ with $k < i$, $i+1 < l$;

4) $S$ contains subsets $\{i,k\},\{i+1,l\}$ with $i+1 < k < l$;

5) $S$ contains subsets $\{k,i\},\{l,i+1\}$ with $k < l < i$;

6) $S$ contains the subsets $\{i\},\{i+1\}$.

\noindent It suffices to show that for every $x \in V_{i+1},y \in J$ we have $x \bullet y \in V_{i-1}$. Set $k = i$ in cases 1, 6 and $l = i+1$ in cases 2, 6 above. Then we have $k < l$ and $(\spa\{e_{i+1}\})^{\perp} = \spa\{e_1,\dots,e_{l-1},e_{l+1},\dots,e_n\}$ in all cases. It follows that $\gamma(e_{i+1} \bullet y,e_k) = \gamma(e_k \bullet y,e_{i+1}) = 0$, because $e_k \bullet y \in \spa\{e_1,\dots,e_{k-1}\} \subset (\spa\{e_{i+1}\})^{\perp}$. But then $e_{i+1} \bullet y \in (\spa\{e_k\})^{\perp} = \spa\{e_1,\dots,e_{i-1},e_{i+1},\dots,e_n\}$. On the other hand, $e_{i+1} \bullet y \in V_i$, and hence $e_{i+1} \bullet y \in V_i \cap (\spa\{e_k\})^{\perp} = V_{i-1}$. Now every vector $x \in V_{i+1}$ can be decomposed as $x = \alpha e_{i+1} + x'$ for some $\alpha \in \mathbb R$ and $x' \in V_i$. But $x' \bullet y \in V_{i-1}$, and hence $x \bullet y \in V_{i-1}$, which proves our claim.

Note that permutations of the coordinates do not change the number of subsets in the partition $S$, in particular, the number of subsets of size 2 stays equal to $m$. Let $j \geq 1$ be the smallest integer such that $\{j,n+1-j\} \not\in S$. If $j > m$, then $S$ is already in the required form. Let now $j \leq m$.

Suppose first that $\{j\} \in S$. Let $i \geq j$ be the smallest integer such that $\{i+1\} \not\in S$. Then there exists $l > i+1$ such that $\{i+1,l\} \in S$ and case 1 above holds. Let us exchange the basis vectors $e_i,e_{i+1}$ and update the partition $S$ accordingly. As a result, we will now have $\{i,l\} \in S$. If $i > j$, then also $\{i-1\} \in S$ and we may exchange $e_{i-1},e_i$. Continuing this process, we arrive at a partition $S$ containing the subset $\{j,l\}$.

We may hence assume that $\{j\} \not\in S$, and there exists $i > j$ such that $\{j,i\} \in S$. Suppose that $i < n+1-j$. Setting $k = j$, we are then in one of the cases 2, 3, 5 above. Let us then exchange the basis vectors $e_i,e_{i+1}$ and update $S$ accordingly. Then we will have $\{j,i+1\} \in S$. If $i+1 < n+1-j$, then we exchange $e_{i+1},e_{i+2}$ and continue the process until $\{j,n+1-j\} \in S$.

We have shown that we can find a sequence of coordinate permutations such that the subset $\{j,n+1-j\}$ becomes an element of $S$, while $(J,\gamma)$ remains in semi-canonical form. Increasing $j$ and repeating this process leads eventually to the required partition $S$. By virtue of case 6 above we may finally sort the remaining indices $k = m+1,\dots,n-m$ according to the value of $\gamma(e_k,e_k)$, bringing $\gamma$ to the desired form.
\end{proof}

{\lemma If the partition $S$ in the previous theorem contains a subset of cardinality one, then the coordinate transformation in this theorem can be chosen to be unimodular. }

\begin{proof}
The determinant of $\gamma$ had modulus 1 at the beginning as well as at the end of the procedure. Hence the applied coordinate transformation had determinant $\pm1$. If the determinant happens to be equal to $-1$, we may apply an additional transformation which does not change the form of $\gamma$. If $S$ contains an index set of cardinality $1$, say $\{i\}$, this transformation can be defined by negating the unit basis vector $e_i$.
\end{proof}

It should be stressed that the semi-canonical form in Definition \ref{can_J} is not unique. It depends on the particular completion of the central ascending series, and also on the choice of the vector $\tilde v$ in case 2 in the proof of Lemma \ref{can_form}.

{\corollary \label{convex_IAHS} Let $M$ be an irreducible improper affine hypersphere with parallel cubic form and definite affine metric. Then $M$ is 1-dimensional and must be a parabola. }

\begin{proof}
Let $(J,\gamma)$ be the real metrised Jordan algebra defined by $M$ as in Theorem \ref{main1}. By Lemma \ref{AHS} $J$ is nilpotent, and by Theorem \ref{th_AHS2} $(J,\gamma)$ is indecomposable. By Theorem \ref{th_can} we can assume that $(J,\gamma)$ is in semi-canonical form. But then the matrix of $\gamma$ must be diagonal, because $\gamma$ is a definite quadratic form. By Lemma \ref{semi_can_red} we have $\dim J = 1$. But a nilpotent algebra of dimension 1 must be the algebra with zero structural tensor. The claim of the corollary now follows from Theorem \ref{th_AHS2}.
\end{proof}

{\lemma \label{sign_dim2} Let $(J,\gamma)$ be a real nilpotent metrised Jordan algebra of dimension $n \geq 2$, such that $\gamma$ has a signature with $k$ negative eigenvalues. If in any semi-canonical form of $(J,\gamma)$ the number of subsets of size 2 in the partition $S$ satisfies $m < -\frac32+\sqrt{\frac94+2\max(k,n-k)}$, then $(J,\gamma)$ is decomposable. }

\begin{proof}
Assume the conditions of the lemma. By Theorem \ref{th_can} we can assume that the partition $S$ is given by $\{\{1,n\},\{2,n-1\},\dots,\{m,n+1-m\},\{m+1\},\dots,\{n-m\}\}$. If $m+1 > n-m$, then $m = k = \frac{n}{2}$, and the inequality yields $n^2+2n < 0$, a contradiction. Hence $m+1 \leq n-m$. The restriction of $\gamma$ to the $(n-2m)$-dimensional subspace $V = \spa\{e_{m+1},\dots,e_{n-m}\}$ has signature $(n-k-m,k-m)$. An isotropic subspace $L \subset V$ can therefore have dimension at most $\min(k-m,n-k-m)$. Thus every subspace of $V$ of dimension $d > \min(k-m,n-k-m)$ must contain a vector $v$ such that $\gamma(v,v) \not= 0$.

Let $m+1 \leq \delta \leq n-m$, set $\sigma = \gamma(e_{\delta},e_{\delta}) = \pm1$, and suppose that the structure coefficient $K^{\delta}_{\alpha\beta}$ of $J$ is nonzero. Define $\alpha' = \alpha$ if $m+1\leq\alpha\leq n-m$, and $\alpha' = n+1-\alpha$ otherwise. Then we have $(\spa\{e_{\alpha}\})^{\perp} = \spa\{e_1,\dots,e_{\alpha'-1},e_{\alpha'+1},\dots,e_n\}$. We obtain $K^{\delta}_{\alpha\beta} = \sigma\gamma(e_{\alpha} \bullet e_{\beta},e_{\delta}) = \sigma\gamma(e_{\delta} \bullet e_{\beta},e_{\alpha}) = \sigma K^{\alpha'}_{\beta\delta} \not= 0$. By property i) of Definition \ref{can_J} it follows that $\alpha' < \delta < \alpha$. Hence $\alpha \geq n-m+1$. Similarly, $\beta \geq n-m+1$. Therefore there exist at most $\frac{m(m+1)}{2}$ index pairs $(\alpha,\beta)$ with $\alpha \leq \beta$ such that $K^{\delta}_{\alpha\beta} \not= 0$, and the set of these pairs is the same for all $\delta = m+1,\dots,n-m$. Thus the dimension of the intersection $(J \bullet J)^{\perp} \cap V$ is at least $\dim V - \frac{m(m+1)}{2} = n - \frac{m(m+5)}{2}$.

The inequality in the lemma yields $\frac{m(m+3)}{2} < \max(k,n-k)$, which is equivalent to $n - \frac{m(m+5)}{2} > \min(k-m,n-k-m)$. By the preceding two paragraphs there exists a vector $v \in (J \bullet J)^{\perp} \cap V$ such that $\gamma(v,v) \not= 0$. Since $\gamma(x \bullet y,v) = 0$ for all $x,y \in J$, the subspace $(\spa\{v\})^{\perp}$ is an ideal. By $n \geq 2$ this ideal is not zero. On the other hand, by \eqref{3symmetry} we have $\gamma(x \bullet v,y) = 0$ for all $x,y \in J$, and hence $x \bullet v = 0$ for all $x \in J$. It follows that $\spa\{v\}$ is also an ideal. But $v \not\in (\spa\{v\})^{\perp}$, and $(J,\gamma)$ is a direct sum of two mutually $\gamma$-orthogonal ideals. This proves the lemma.
\end{proof}

{\corollary \label{sign_dim} Let $(J,\gamma)$ be an indecomposable real nilpotent metrised Jordan algebra of dimension $n \geq 2$, such that $\gamma$ has a signature with $k$ negative (or $k$ positive) eigenvalues. Then $n \leq \frac{k(k+5)}{2}$. }

\begin{proof}
By Lemma \ref{sign_dim2} the number $m$ of subsets of size 2 in the partition $S$ corresponding to any semi-canonical form of $(J,\gamma)$ must be at least $-\frac32+\sqrt{\frac94+2(n-k)}$. On the other hand, it cannot exceed $k$. It follows that $k \geq -\frac32+\sqrt{\frac94+2(n-k)}$, which readily yields the required bound on $n$.
\end{proof}

{\corollary \label{AHS_sign_dim} Let $M$ be an irreducible improper affine hypersphere with parallel cubic form and such that its affine metric has a signature with $k$ negative (or $k$ positive) eigenvalues. Then the dimension of $M$ is bounded from above by $n \leq \max(1,\frac{k(k+5)}{2})$. }

\begin{proof}
Apply Corollary \ref{sign_dim} to the indecomposable nilpotent metrised Jordan algebra $(J,\gamma)$ defined by $M$ as in Theorem \ref{main1}.
\end{proof}

\subsection{Cayley hypersurfaces} \label{subs_Cayley}

In this subsection we provide a simple representation of the metrised Jordan algebras corresponding to the Cayley hypersurfaces. The $n$-dimensional Cayley hypersurface is given by the graph of the function \cite{EE06}
\begin{equation} \label{def_Cayley}
F(x_1,\dots,x_n) = \sum_{d=2}^{n+1} \frac{(-1)^d}{d} \sum_{i_1+\dots+i_d = n+1} \prod_{j = 1}^d x_{i_j}.
\end{equation}
Here the sum over the $d$-tuple $(i_1,\dots,i_d)$ counts all permutations of the indices, e.g., for $n = 3$ each of the triples $(1,1,2),(1,2,1),(2,1,1)$ makes a contribution to this sum. The Cayley hypersurfaces possess a commutative subgroup of the automorphism group which acts transitively \cite[Prop.~2]{EE06}, they are improper affine hyperspheres \cite[Prop.~4]{EE06} with parallel cubic form and flat affine metric \cite{HLZ11}. We now give the following description of the metrised Jordan algebras generated by the Cayley hypersurfaces.

{\theorem \label{th_Cayley} Let $M \subset \mathbb R^{n+1}$ be an $n$-dimensional Cayley hypersurface. Then the metrised Jordan algebra $(J,\gamma)$ defined by $M$ as in Theorem \ref{main1} can be described as the polynomial quotient ring $(t\mathbb R[t])/(t^{n+1})$. If $p,q$ are two polynomials in $t\mathbb R[t]$, then $\gamma([p],[q])$ is defined as the coefficient at $t^{n+1}$ of the product $p \cdot q$. Here $[p],[q] \in J$ are the equivalence classes of $p,q$. }

\begin{proof}
We shall explicitly compute the matrix of $\gamma$ and the coefficients $K^{\delta}_{\alpha\beta}$ of the structure tensor of $J$ and compare them to the second and third derivatives at the origin of the function $F$ defined in \eqref{def_Cayley}.

First note that if $p,q \in t\mathbb R[t],r \in \mathbb R[t]$ are polynomials, then $(rt^{n+1}) \cdot q$ is divisible by $t^{n+2}$ and hence the products $p \cdot q,(p+rt^{n+1}) \cdot q$ have the same coefficient at $t^{n+1}$. Thus $\gamma([p],[q])$ is well-defined. For $p = \sum_{k=1}^n p_kt^k$, $q = \sum_{k=1}^n q_kt^k$ we have
\[ p \cdot q = \sum_{k,l=1}^n p_kq_lt^{k+l} = \sum_{m=2}^{2n} \left( \sum_{k=1}^{\min(n,m-1)} p_kq_{m-k} \right) t^m.
\]
In particular, for $m \leq n+1$ the coefficient of $p\cdot q$ at $t^m$ is given by $\sum_{k=1}^{m-1} p_kq_{m-k}$. It follows that
\[ \gamma([p],[q]) = \sum_{k=1}^n p_kq_{n+1-k},\qquad [p] \bullet [q] = \left[ \sum_{m=2}^n \left( \sum_{k=1}^{m-1} p_kq_{m-k} \right) t^m \right].
\]
The elements of $\gamma,K$, and their contraction are hence given by
\[ \gamma_{\alpha\beta} = \left\{ \begin{array}{rl} 0, & \alpha+\beta \not= n+1, \\ 1, & \alpha+\beta = n+1, \end{array} \right. \quad
K^{\delta}_{\alpha\beta} = \left\{ \begin{array}{rl} 0, & \alpha+\beta \not= \delta, \\ 1, & \alpha+\beta = \delta, \end{array} \right. \quad \gamma_{\delta\rho}K^{\rho}_{\alpha\beta} = \left\{ \begin{array}{rl} 0, & \alpha+\beta+\delta \not= n+1, \\ 1, & \alpha+\beta+\delta = n+1. \end{array} \right.
\]
On the other hand, the second and third derivative of $F$ at the origin are given by
\[ F_{\alpha\beta} = \left\{ \begin{array}{rl} 0, & \alpha+\beta \not= n+1, \\ 1, & \alpha+\beta = n+1, \end{array} \right. \qquad\qquad
F_{\alpha\beta\delta} = \left\{ \begin{array}{rl} 0, & \alpha+\beta+\delta \not= n+1, \\ -2, & \alpha+\beta+\delta = n+1. \end{array} \right.
\]
The last relation in \eqref{correspondence} completes the proof.
\end{proof}

{\corollary \label{Cayley_irr} The Cayley hypersurfaces are irreducible. }

\begin{proof}
By Lemmas \ref{decomp_lemma} and \ref{direct_sums} the Cayley hypersurfaces are irreducible if and only if the corresponding metrised Jordan algebras $(J,\gamma)$ are indecomposable. Suppose that $(J,\gamma)$ can be decomposed into two non-trivial factors. Then each of the factors is nilpotent and there exist at least two linearly independent elements $p,p'$ such that $p \bullet q = p' \bullet q = 0$ for all $q \in J$. But in the polynomial quotient ring $(t\mathbb R[t])/(t^{n+1})$ any such element must be proportional to $[t^n]$, a contradiction.
\end{proof}

From the expression in Theorem \ref{th_AHS2} it can be seen that the function $F$, evaluated at the equivalence class $[p]$ of a polynomial $p \in t\mathbb R[t]$, is given by the coefficient at $t^{n+1}$ of the Taylor expansion around $t = 0$ of the function $f(t) = p(t) - \log(1+p(t))$. This function $f = f(t;p_1,p_2\dots)$ is hence the generating function of the functions $F$ defining the Cayley hypersurfaces in all dimensions.

The Jordan algebras corresponding to the Cayley hypersurfaces involve univariate polynomials. Other improper affine hyperspheres with flat affine metric can be constructed using nilpotent metrised associative algebras emanating from multivariate polynomials. Consider, e.g., the polynomial quotient ring $J = (t\mathbb R[t,s])/(\{t^3,s^2\})$, with elements being the equivalence classes of $p(t,s) = p_1ts + p_2t^2 + p_3t + p_4t^2s$. For two polynomials $p,q \in t\mathbb R[t,s]$, define $\gamma([p],[q])$ as the coefficient at $t^3s$ of the product $p \cdot q$. Then $(J,\gamma)$ is a nilpotent metrised associative algebra, and the function from Theorem \ref{th_AHS2} is given by $F = p_1p_2 + p_3p_4 - p_1p_3^2$. It is not hard to check that the improper affine hypersphere defined by the graph of this function is isomorphic to that in line 3 of \eqref{class4}.

\subsection{Low-dimensional improper affine hyperspheres} \label{subs_lowdim}

In this subsection we classify all irreducible improper affine hyperspheres with parallel cubic form up to dimension 5. Since the classification of improper affine hyperspheres with parallel cubic form is known in dimensions 2, 3, 4 \cite{MagidNomizu89},\cite{HuLi11},\cite{HLLV11b}, we need to check only for reducibility in these dimensions.

{\theorem \label{th_class234} Let $M$ be an irreducible improper affine hypersphere with parallel cubic form in dimension $n \leq 4$. If $n = 1$, then $M$ is a quadric. If $n = 2,3$, then $M$ is isomorphic to the $n$-dimensional Cayley hypersurface. If $n = 4$, then $M$ is isomorphic to exactly one of the hypersurfaces in lines 2, 3, 6 of classification \eqref{class4}. }

\begin{proof}
For $n = 1$ the claim follows from Corollary \ref{convex_IAHS}. For $n \geq 2$ $M$ cannot be a quadric, because a quadric is irreducible if and only if its dimension is 1.

The first hypersurface in \eqref{class3} is reducible, because it decomposes as $w = (xy-\frac13y^3)+(\frac12z^2)$. For $n = 2,3$ the claim now follows from Corollary \ref{Cayley_irr}.

The hypersurfaces in lines 1, 4, 5, 7, 8 of \eqref{class4} can be decomposed non-trivially as
\begin{eqnarray*}
x_5 &=& (x_1x_2-\frac13x_1^2)+(x_3x_4-\frac13x_3^2), \\
x_5 &=& (x_1x_3-x_1^2x_2+\frac14x_1^4+\frac12x_2^2)-(\frac12x_4^2), \\
x_5 &=& (x_1x_2-\frac13x_1^3)+(\frac12x_3^2)\mp(\frac12x_4^2), \\
x_5 &=& (x_1x_3+\frac12x_2^2-x_1^2x_2+\frac14x_1^4)+(\frac12x_4^2).
\end{eqnarray*}
Let $(J,\gamma)$ be the real nilpotent metrised Jordan algebra defined by the hypersurface in line 2, 3, or 6 of \eqref{class4} as in Theorem \ref{main1}. After a suitable permutation of the coordinates, $(J,\gamma)$ is in semi-canonical form given by
\[ K^2_{33} = a,\ K^2_{34} = K^2_{43} = K^1_{33} = b,\ K^2_{44} = K^1_{34} = K^1_{43} = c,\ K^1_{44} = d,\ K^1_{24} = K^1_{42} = K^3_{44} = e,
\]
\begin{equation} \label{algebras4}
\gamma = \begin{pmatrix} 0 & 0 & 0 & 1 \\ 0 & 0 & 1 & 0 \\ 0 & 1 & 0 & 0 \\ 1 & 0 & 0 & 0 \end{pmatrix},\quad \begin{pmatrix} a \\ b \\ c \\ d \\ e \end{pmatrix} = \begin{pmatrix} 0 \\ \frac13 \\ \frac13 \\ 0 \\ 0 \end{pmatrix},\begin{pmatrix} 0 \\ 0 \\ \frac13 \\ 0 \\ 0 \end{pmatrix},\ \mbox{or}\ \begin{pmatrix} 0 \\ 1 \\ 0 \\ 0 \\ 1 \end{pmatrix},
\end{equation}
all other elements of the structure tensor $K$ being zero. Let us show that all three algebras are indecomposable. If the 1-dimensional zero algebra is present as a factor, then the operators $L_x$, $x \in J$, must have a common eigenvector $v$ with eigenvalue zero and such that $\gamma(v,v) \not= 0$. Such an eigenvector does not exist in any of the cases. Hence if any of the algebras in \eqref{algebras4} is decomposable, then it can only split in two indecomposable factors of dimension 2. These factors must be isomorphic to the algebra defined by the Cayley surface, because there is no other irreducible improper affine hypersphere with parallel cubic form in dimension 2. In order to show that the product of two such 2-dimensional algebras is not isomorphic to one of the algebras \eqref{algebras4}, we consider the set of vectors $x = (x_1,x_2,x_3,x_4)^T \in J$ where the operator $L_x$ drops rank with respect to its generic rank. In the product the operator $L_x$ drops rank on two distinct hyperplanes, but in the three algebras above it drops rank at the sets $\{ x \in J \,|\, x_3 = x_4 = 0\}$ for the first and $\{ x \in J \,|\, x_4 = 0\}$ for the other two. Hence such a product cannot be isomorphic to one of the algebras \eqref{algebras4}, and they must be indecomposable. But then the hypersurfaces in lines 2, 3, 6 of \eqref{class4} are irreducible by Lemma \ref{decomp_lemma}. This completes the proof.
\end{proof}

Note also that all algebras \eqref{algebras4} are actually associative, and hence the affine pseudo-metric of the corresponding improper affine hyperspheres is flat.

{\corollary \label{Lorentz_IAHS} Let $M$ be an irreducible improper affine hypersphere with parallel cubic form and with Lorentzian affine fundamental form. Then $\dim M = 2$ or $\dim M = 3$ and $M$ is isomorphic to the Cayley hypersurface in the corresponding dimension. }

\begin{proof}
By Corollary \ref{AHS_sign_dim} the dimension of $M$ does not exceed 3. Clearly it must be at least 2. The proof is concluded by application of Theorem \ref{th_class234}, noting that the Cayley surface and its 3-dimensional analog have a Lorentzian signature.
\end{proof}

A general, i.e., not necessarily irreducible, improper affine hypersphere with parallel cubic form and with Lorentzian affine metric must be either a quadric or a direct product of an irreducible such improper affine hypersphere and a convex one. We thus recover the classification of Lorentzian improper affine hyperspheres with parallel cubic form which was derived in \cite{HLLV11a}.

\medskip

In order to classify all irreducible improper affine hyperspheres with parallel cubic form in dimension 5, we shall consider the corresponding nilpotent metrised Jordan algebras.

{\theorem \label{algebra5} An indecomposable real nilpotent metrised Jordan algebra $(J,\gamma)$ of dimension 5, with $\gamma$ having signature $(+++--)$, is isomorphic by a unimodular isomorphism to exactly one of the metrised algebras given by
\[ K^4_{55} = K^1_{25} = K^1_{52} = a,\ K^3_{44} = K^2_{34} = K^2_{43} = b,\ K^3_{45} = K^3_{54} = K^2_{35} = K^2_{53} = K^1_{34} = K^1_{43} = c,
\]
\[ K^3_{55} = K^1_{35} = K^1_{53} = d,\ K^2_{44} = e,\ K^1_{44} = K^2_{45} = K^2_{54} = f,\ K^2_{55} = K^1_{45} = K^1_{54} = g,\ K^1_{55} = h,
\]
\begin{equation} \label{algebras5}
\gamma = \begin{pmatrix} 0 & 0 & 0 & 0 & 1 \\ 0 & 0 & 0 & 1 & 0 \\ 0 & 0 & 1 & 0 & 0 \\ 0 & 1 & 0 & 0 & 0 \\ 1 & 0 & 0 & 0 & 0 \end{pmatrix},\quad \begin{pmatrix} a \\ b \\ c \\ d \\ e \\ f \\ g \\ h \end{pmatrix} = \begin{pmatrix} 0 \\ \epsilon \\ 0 \\ 0 \\ 0 \\ 0 \\ 1 \\ 0 \end{pmatrix},\begin{pmatrix} 0 \\ \epsilon \\ 0 \\ \epsilon \\ 0 \\ 0 \\ 0 \\ 0 \end{pmatrix},\begin{pmatrix} 0 \\ \epsilon \\ 0 \\ \epsilon \\ \alpha \\ 0 \\ 0 \\ 0 \end{pmatrix},\begin{pmatrix} 0 \\ 0 \\ 1 \\ 0 \\ 0 \\ 0 \\ 0 \\ 0 \end{pmatrix},\begin{pmatrix} 0 \\ 0 \\ 1 \\ 0 \\ 1 \\ 0 \\ 0 \\ 0 \end{pmatrix},\begin{pmatrix} 0 \\ 0 \\ 1 \\ 0 \\ \alpha \\ 0 \\ 0 \\ \epsilon\alpha \end{pmatrix},\ \mbox{or}\ \begin{pmatrix} 1 \\ 0 \\ 1 \\ 0 \\ 1 \\ 0 \\ 0 \\ 0 \end{pmatrix},
\end{equation}
all other elements of the structure tensor $K$ being zero. Here $\alpha > 0$ is a parameter, and $\epsilon = \pm1$. }

\begin{proof}
From Lemma \ref{sign_dim2} it follows that any semi-canonical form of $(J,\gamma)$ has a partition $S$ with two blocks of size 2. By Theorem \ref{th_can} we may assume that $\gamma$ has the required form. Property i) in Definition \ref{can_J} and \eqref{3symmetry} then imply that the structure tensor $K$ has the required form, for some $a,\dots,h$. Using unimodular linear transformations of the vector space underlying $J$ we shall now bring the parameters $a,\dots,h$ to a canonical form. By direct verification it can be established that $J$ is a Jordan algebra if and only if $a = 0$ or $b = ae-c^2 = 0$. We consider these cases separately.

1. $a = 0$. For $A \in GL(2,\mathbb R)$ the transformation $\diag(A,1,(PAP)^{-T})$ of $J$, where $P$ is the nontrivial $2 \times 2$ permutation matrix, acts on the parameters $b,c,d$ in a way isomorphic to the action of $A$ on the coefficients of the quadratic form $bx_1^2+2cx_1x_2+dx_2^2$. This action has 6 orbits, corresponding to the different possible signatures of the form. If $b = c = d = 0$, then $\spa\{e_1,e_2,e_4,e_5\}$ and $\spa\{e_3\}$ are $\gamma$-orthogonal ideals, and $J$ is decomposable. The other orbits have representatives for which $(b,c,d)$ equals $(\epsilon,0,0),(\epsilon,0,\epsilon),(0,1,0)$, respectively.

1.1. $(b,c,d) = (\epsilon,0,0)$. Changing the sign of the second and fourth coordinate induces the transformation $(e,f,g,h) \mapsto (-e,f,-g,h)$. Changing the sign of the first and last coordinate induces the transformation $(e,f,g,h) \mapsto (e,-f,g,-h)$. Consider the Lie algebra of elements of the form $\begin{pmatrix} a_{11} & 0 & a_{13} &    a_{14} &    0 \\ a_{21} & 0 & a_{23} &    0 &    -a_{14} \\   0 & 0 &   0 & -a_{23} & -a_{13} \\   0 & 0 &   0 &  0 &    0 \\   0 & 0 &   0 & -a_{21} & -a_{11} \end{pmatrix}$. The action of the group generated by this Lie algebra and the sign changes on the space of parameters $(e,f,g,h)$ has three orbits, given by $\{(e,f,g,h) \,|\, h \not= 0\}$, $\{(e,f,g,h) \,|\, g \not= 0,h = 0\}$, and $\{(e,f,g,h) \,|\, g = h = 0\}$, respectively. The first and the last orbit have representatives given by $(e,f,g,h) = (0,0,0,1),(0,0,0,0)$, respectively. For the corresponding algebras the subspaces $\spa\{e_1,e_5\}$, $\spa\{e_2,e_3,e_4\}$ are mutually $\gamma$-orthogonal ideals, and $J$ is decomposable. The remaining orbit has the representative $(e,f,g,h) = (0,0,1,0)$, giving the first column in \eqref{algebras5}.

1.2. $(b,c,d) = (\epsilon,0,\epsilon)$. Consider the Lie algebra of elements of the form $\begin{pmatrix} 0 & a_{12} & a_{13} & a_{14} & 0 \\ -a_{12} & 0 & a_{23} & 0 & -a_{14} \\ 0 & 0 & 0 & -a_{23} & -a_{13} \\ 0 & 0 & 0 & 0 & -a_{12} \\ 0 & 0 & 0 & a_{12} & 0 \end{pmatrix}$. The action of the group generated by this Lie algebra on the space of parameters $(e,f,g,h)$ has the invariant $(e-3g)^2+(h-3f)^2$, and every level set of this invariant is an orbit. These orbits have representatives $(e,f,g,h) = (\alpha,0,0,0)$, where $\alpha \geq 0$ is a parameter. This yields the second and third column in \eqref{algebras5}.

1.3. $(b,c,d) = (0,1,0)$. Changing the sign of all coordinates except the third induces the transformation $(e,f,g,h) \mapsto (-e,-f,-g,-h)$. The transformation given by the matrix $\diag(P,1,P)$, where $P$ is the nontrivial $2 \times 2$ permutation matrix, acts like $(e,f,g,h) \mapsto (h,g,f,e)$. Consider the Lie algebra of elements of the form $\begin{pmatrix} a_{11} & 0 & a_{13} & a_{14} & 0 \\ 0 & -a_{11} & a_{23} & 0 & -a_{14} \\ 0 & 0 & 0 & -a_{23} & -a_{13} \\ 0 & 0 & 0 & a_{11} & 0 \\ 0 & 0 & 0 & 0 & -a_{11} \end{pmatrix}$. The action of the group generated by this Lie algebra on the space of parameters $(e,f,g,h)$ has the invariant $e\cdot h$. The two-dimensional subspace given by $e = h = 0$ is an orbit of this  action, and all other orbits have dimension 3. Taking into account the discrete transformations considered above, we may restrict our consideration to orbits of elements satisfying $e \geq |h|$. These orbits have representatives $(e,f,g,h) = (0,0,0,0),(1,0,0,0),(\alpha,0,0,\epsilon\alpha)$, where $\alpha > 0$ is a parameter, and $\epsilon = \pm1$. This yields columns 4, 5, 6 in \eqref{algebras5}.

2. $a \not= 0$, $b = 0$, $ae = c^2$. The transformation given by the matrix $\diag(1,\alpha^{-1},1,\alpha,1)$, $\alpha \not= 0$, induces the transformation $a \mapsto \alpha a$. We can hence assume $a = 1$. The transformation given by the matrix $\diag(\alpha,1,1,1,\alpha^{-1})$, $\alpha \not= 0$, induces the transformation $c \mapsto \alpha c$. We can hence assume $c = 0$ or $c = 1$, which we consider separately.

2.1. $c = e = 0$. The one-parametric group of transformations $\diag\left(1,\begin{pmatrix} 1 & \tau & -\frac{\tau^2}{2} \\ 0 & 1 & -\tau \\ 0 & 0 & 1 \end{pmatrix},1\right)$ induces the transformation $d \mapsto d + \tau$, and we can assume $d = 0$. But then $J$ is decomposable, because $\spa\{e_1,e_2,e_4,e_5\}$ and $\spa\{e_3\}$ are $\gamma$-orthogonal ideals.

2.2. $c = e = 1$. Consider the Lie algebra of elements of the form $\begin{pmatrix} 0 & a_{12} & a_{13} &  a_{14} &    0 \\  0 &   0 & a_{23} &    0 & -a_{14} \\ 0 &   0 &   0 & -a_{23} & -a_{13} \\ 0 &   0 &   0 &    0 & -a_{12} \\ 0 &   0 &   0 &    0 &    0 \end{pmatrix}$. The group generated by this algebra acts transitively on the space of parameters $(d,f,g,h)$. Hence we can set $(d,f,g,h) = (0,0,0,0)$, yielding the last column in \eqref{algebras5}.

We have shown that $(J,\gamma)$ is isomorphic to at least one of the metrised algebras in \eqref{algebras5}. In order to show that these are mutually non-isomorphic, we list for each whether it is associative, the dimension of its automorphism group, and for some of them the set of vectors $x = (x_1,x_2,x_3,x_4,x_5)^T$ where the operator $L_x$ or $U_x$ drops rank with respect to its generic rank:

\medskip

\begin{tabular}{c|ccccccccc}
algebra & 1 & 2 & 3 & 4 & 5 & 6 & 7 \\
\hline
associative & X & & & & & & X \\
$\dim\Aut(J,\gamma)$ & 2 & 2 & 1 & 2 & 1 & 1 & 0 \\
$L_x$ drops rank & & & & & $x_4(x_4^2\!-\!2x_3x_5)\!=\! 0$ & $\frac{\alpha}{2}(x_4^3\!+\!\epsilon x_5^3)\!=\!x_3x_4x_5$ & \\
$U_x$ drops rank & & $x_4\!=\!x_5\!=\!0$ & $x_4\!=\!x_5\!=\!0$ & $x_4x_5\!=\!0$ & $x_4x_5\!=\!0$ & $x_4x_5\!=\!0$ &
\end{tabular}

\medskip

\noindent The table shows that algebras in different columns are non-isomorphic. We still need to show that the metrised algebras in columns 3, 6, respectively, are non-isomorphic for different values of the parameter $\alpha$. To this end we consider the invariant subspaces $J^2 = J \bullet J$, $J^3 = J^2 \bullet J$. In both cases $\dim J^2 = 3$, $\dim J^3 = 2$, and $J^3$ is the kernel of the restriction of the form $\gamma$ to $J^2$. Hence $\gamma$ defines a nonzero quadratic form on the quotient space $J^2/J^3$, equipping it with an invariant Euclidean norm. The expression $\frac13\gamma(x \bullet x,x)$ defines an invariant cubic polynomial $p$ on $J$. Since $J^3 \bullet J = 0$, this polynomial is also well-defined on the 3-dimensional quotient space $J/J^3$. The unit sphere in $J^2/J^3$ defines an invariant nonzero vector $v \in J/J^3$ up to a sign. The polynomial $p$ on $J/J^3$ and the vector $v$ are explicitly given by $p(x_3,x_4,x_5) = x_3(x_4^2+x_5^2)+\frac{\alpha}{3}x_4^3,2x_3x_4x_5+\frac{\alpha}{3}(x_4^3+\epsilon x_5^3)$ for the algebras in columns 3, 6, respectively, and $v = (\pm 1,0,0)$ in both cases. It is straightforward to check that the pairs $(p,v)$ are non-isomorphic for different values of $\alpha$ in both cases. Finally, for different values of $\epsilon$ the corresponding metrised Jordan algebras are isomorphic, but by a linear isomorphism with determinant $-1$. Under unimodular isomorphisms different values of $\epsilon$ lead to non-isomorphic graph immersions.

Finally, let us show that all algebras are indecomposable. If some algebra is decomposable, then its factors have dimension at most 4 and must be associative. But then their product is also associative, which proves irreducibility of the algebras in columns 2---6. Let us consider the remaining two algebras. The zero algebra cannot be present as a factor for the same reason as in the proof of Theorem \ref{th_class234}. Hence the algebras could only split in two indecomposable factors of dimensions 2 and 3. These factors must be isomorphic to the metrised algebras defined by the Cayley hypersurfaces in the corresponding dimension. But the product of these two metrised algebras has an automorphism group of dimension 1, which shows that the algebras in columns 1 and 7 are also indecomposable. This completes the proof.
\end{proof}


{\corollary \label{class5} Let $M$ be a 5-dimensional irreducible improper affine hypersphere with parallel cubic form. Then either $M$ is isomorphic by a unimodular transformation to the graph of exactly one of the following functions:
\begin{eqnarray*}
F &=& x_1x_5 + x_2x_4 + \frac12x_3^2 - \epsilon x_3x_4^2 - x_4x_5^2 + \frac14x_4^4, \\
F &=& x_1x_5 + x_2x_4 + \frac12x_3^2 - \epsilon x_3x_4^2 - \epsilon x_3x_5^2 + \frac14x_4^4 + \frac12x_4^2x_5^2 + \frac14x_5^4, \\
F &=& x_1x_5 + x_2x_4 + \frac12x_3^2 - \epsilon x_3x_4^2 - \epsilon x_3x_5^2 - \frac{\alpha}{3}x_4^3 + \frac14x_4^4 + \frac12x_4^2x_5^2 + \frac14x_5^4, \\
F &=& x_1x_5 + x_2x_4 + \frac12x_3^2 - 2x_3x_4x_5 + x_4^2x_5^2, \\
F &=& x_1x_5 + x_2x_4 + \frac12x_3^2 - 2x_3x_4x_5 - \frac13x_4^3 + x_4^2x_5^2, \\
F &=& x_1x_5 + x_2x_4 + \frac12x_3^2 - 2x_3x_4x_5 - \frac{\alpha}{3}(x_4^3 + \epsilon x_5^3) + x_4^2x_5^2, \\
F &=& x_1x_5 + x_2x_4 + \frac12x_3^2 - 2x_3x_4x_5 - x_2x_5^2 - \frac13x_4^3 + x_3x_5^3 + \frac32x_4^2x_5^2 - x_4x_5^4 + \frac16x_5^6,
\end{eqnarray*}
where $\alpha > 0$ is a real parameter and $\epsilon = \pm1$, or $M$ is isomorphic to the graph of $-F$, where $F$ is exactly one of the functions listed above. The affine metric is flat in the first and in the last of the seven listed cases, in all other cases it has non-vanishing curvature. }

\begin{proof}
An irreducible improper affine hypersphere with parallel cubic form in dimension $n \geq 4$ cannot have a definite or Lorentzian affine metric by Corollaries \ref{convex_IAHS} and \ref{Lorentz_IAHS}. We may hence assume that the affine fundamental form of $M$ has either signature $(+++--)$ or signature $(++---)$. Let $(J,\gamma)$ be the indecomposable nilpotent metrised Jordan algebra defined by $M$ as in Theorem \ref{main1}. If the signature of $\gamma$ is $(+++--)$, then by Theorem \ref{algebra5} $(J,\gamma)$ is isomorphic to exactly one of the metrised algebras in \eqref{algebras5}. Evaluating the formula for the defining function $F$ in Theorem \ref{th_AHS2} for these algebras leads to the list in the corollary. If the signature of $\gamma$ is $(++---)$, then by negating the transversal vector field $\xi$ we may reduce the classification to the previous case. However, the sign change in $\xi$ leads to a sign change in the defining function $F$.

The flatness of the affine metric is equivalent to associativity of the corresponding Jordan algebra and holds in the first and the last of the listed cases. This completes the proof.
\end{proof}

Thus 5 is the smallest dimension in which improper affine hyperspheres with parallel cubic form and non-vanishing curvature of the affine pseudo-metric exist. It is also the smallest dimension where continuous families of mutually non-isomorphic such hypersurfaces occur.

\section{Other classes of graph immersions} \label{sec_other}

Up to now we have considered non-degenerate graph immersions with parallel cubic form in the sense that $\hat\nabla C = 0$, where $\hat\nabla$ is the Levi-Civita connection of the affine pseudo-metric. In this section we apply the developed algebraic methods to graph immersions or improper affine hyperspheres whose cubic form or difference tensor is parallel with respect to the affine connection $\nabla$. We determine which classes of algebras arise from these classes of hypersurface immersions. We have the following analog of Theorem \ref{main1}.

{\theorem \label{main1a} Let $f: \Omega \to \mathbb A^{n+1}$ be a non-degenerate graph immersion. Let $y \in \Omega$ be a point and let $\bullet: T_y\Omega \times T_y\Omega \to T_y\Omega$ be the multiplication $(u,v) \mapsto K(u,v)$ defined by the difference tensor $K = \nabla - \hat\nabla$ at $y$. Let $\gamma$ be the symmetric non-degenerate bilinear form defined on $T_y\Omega$ by the affine fundamental form $h$.

Then the tangent space $T_y\Omega$, equipped with the multiplication $\bullet$, is a real commutative algebra $A$, and $(A,\gamma)$ is a metrised algebra, i.e., $\gamma$ is a trace form on $A$. }

\begin{proof}
The proof is the same as that of Theorem \ref{main1}, with the parts concerning the Jordan identity dropped.
\end{proof}

We now consider several classes of hypersurface immersions.

\subsection{Graph immersions with $\nabla K = 0$}

We shall prove the following result.

{\theorem \label{DKint} Let $f: \Omega \to \mathbb A^{n+1}$ be a non-degenerate graph immersion satisfying $\nabla K = 0$, and let $y \in \Omega$. Let $(A,\gamma)$ be the metrised commutative algebra defined in Theorem \ref{main1a}. Then $A$ is associative and $f$ has representing function $F$ locally around $y$ given by $F(x) = \sum_{k=2}^{\infty} \frac{(-2)^{k-2}}{k!}\gamma(x,x^{k-1})$.

On the other hand, let $(A,\gamma)$ be a metrised associative, commutative algebra. Then the graph of the function $F(x) = \sum_{k=2}^{\infty} \frac{(-2)^{k-2}}{k!}\gamma(x,x^{k-1})$ satisfies $\nabla K = 0$, and at $x = {\bf 0}$ the difference tensor $K$ consists of the structure coefficients of $A$.

Every graph immersion satisfying $\nabla K = 0$ can be extended to a hypersurface embedding with $\nabla K = 0$, given globally by the graph of a function $F: \mathbb R^n \to \mathbb R$.

The graph immersion is an improper affine hypersphere if and only if $\det\gamma = \pm1$ and $A$ is nilpotent. }

\begin{proof}
The condition $\nabla K = 0$ implies that the difference tensor $K$ is constant in any affine chart. On the other hand, $K$ is given by \eqref{correspondence}, where $F$ is a representing function of the graph immersion. The condition $\nabla K = 0$ amounts to $F_{,\alpha\beta\gamma\delta} = F_{,\alpha\beta\rho}F^{,\rho\sigma}F_{,\gamma\delta\sigma}$. By the symmetry of the left-hand side with respect to the four indices we obtain that $K^{\alpha}_{\beta\rho}K^{\rho}_{\gamma\delta} = K^{\alpha}_{\gamma\rho}K^{\rho}_{\beta\delta}$. This condition is equivalent to the associativity of the algebra $A$.

By \eqref{correspondence} we have $\frac{\partial}{\partial x^{\alpha}}F_{,\beta\gamma} = F_{,\alpha\beta\gamma} = -2F_{,\gamma\delta}K^{\delta}_{\alpha\beta}$. This is a linear differential equation with constant coefficients on the derivative $F''$, with initial condition $F''({\bf 0}) = \gamma$. It can be integrated to $F''(x) = \gamma\exp(-2L_x)$, where the right-hand side is to be understood as a matrix product, and $L_x$ is the operator of multiplication with $x \in A$. Further integration leads to $F(x) = F({\bf 0}) + \langle F'({\bf 0}),x \rangle + x^T\gamma\left( \sum_{k=0}^{\infty} \frac{(-2)^k}{(k+2)!}L_x^k \right)x$, which with $F({\bf 0}) = F'({\bf 0}) = 0$ yields the expression in the formulation of the theorem.

Conversely, let $(A,\gamma)$ be a metrised associative, commutative algebra, and let $F$ be as in the theorem. For $u,v \in A$ we get $\nabla_u\nabla_vF(x) = \sum_{k=2}^{\infty} \frac{(-2)^{k-2}}{(k-2)!}\gamma(u,x^{k-2} \bullet v) = \sum_{k=0}^{\infty} \frac{(-2)^k}{k!}\gamma(u,L_x^kv)$. Here we used in the first equality that $A$ is associative and $\gamma$ is a trace form. It follows that $F''(x) = \gamma\exp(-2L_x)$, which yields $F''({\bf 0}) = \gamma$ and, by the relation $[L_x,L_u] = 0$, $\nabla_uF''(x) = -2\gamma\exp(-2L_x)L_u$. By \eqref{correspondence} we finally obtain $K_u = -\frac12(F'')^{-1}(\nabla_uF'') = L_u$, and hence the difference tensor $K$ is constant and contains the structure coefficients of $A$.

That every graph immersion with $\nabla K = 0$ can be extended to the graph of a function defined globally on $\mathbb R^n$ follows from the fact that $F$ is given by an exponential which is defined for all $x \in A$.

Finally, by Lemma \ref{IAHS} the graph of the function $F$ is an improper affine hypersphere if and only if $\det F'' \equiv \pm1$, which is equivalent to $\det\gamma = \pm1$ and $\det\exp(-2L_x) \equiv 1$. The second condition is equivalent to $\tr L_x = 0$ for all $x \in A$. Let us show that this condition is equivalent to the nilpotency of $A$ (cf.\ also \cite[Lemma 3.3]{DillenVrancken98}). We have $L_{x^k} = L_x^k$ for $k \geq 1$, because $A$ is associative. Hence $\tr L_x = 0$ implies $\tr L_x^k = 0$ for all $k \geq 1$, which in turn implies that $L_x$ is nilpotent. By Lemma \ref{nilnil} the algebra $A$ must then itself be nilpotent. On the other hand, if $A$ is nilpotent, then $\tr L_x = 0$ for all $x \in A$. This completes the proof.
\end{proof}

By Remark \ref{asso_flat} and Theorem \ref{DKint} a graph immersion satisfying $\nabla K = 0$ has necessarily a flat affine metric.

The algebra $A$ does not depend on the chosen point $y \in \Omega$, because the structure tensor $K$ of $A$ is constant. However, the form $\gamma$ depends on $y$, and for different points $y$ the metrised algebras $(A,\gamma)$ may be non-isomorphic, as the following example shows.

Consider the function $F(x_1,x_2) = \frac{\sqrt{2}}{8}e^{-2(x_1+x_2)}\left( \sqrt{2}\cosh(2\sqrt{2}x_2) + \sinh(2\sqrt{2}x_2) \right)$. The difference tensor of the corresponding graph immersion is constant and given by $K^1_{12} = K^1_{21} = K^2_{11} = 0$, $K^1_{11} = K^1_{22} = K^2_{12} = K^2_{21} = 1$, $K^2_{22} = 2$. The algebra $A$ is hence defined by the multiplication $(x_1,x_2) \bullet (y_1,y_2) = (x_1y_1+x_2y_2,x_1y_2+x_2y_1+2x_2y_2)$. This algebra is unital with unit element $e = (1,0)$. The Hessian $F''(y)$, applied to the vector $e$, yields $\frac{\sqrt{2}}{2}e^{-2(y_1+y_2)}\left( \sqrt{2}\cosh(2\sqrt{2}y_2) + \sinh(2\sqrt{2}y_2) \right)$. The value of the form $\gamma$ on the unit element of $A$ is an invariant of the metrised algebra $(A,\gamma)$, however, and cannot depend on $y$ if the isomorphism class of $(A,\gamma)$ does not. Hence this graph immersion defines non-isomorphic metrised algebras at different points.

Theorem \ref{DKint} suggests a close relation between non-degenerate graph immersions satisfying $\nabla K = 0$ on the one hand, and non-degenerate graph immersions with flat affine metric and satisfying $\hat\nabla C = 0$ on the other hand. Namely, both classes of hypersurface immersions generate the same class of metrised algebras. Given a non-degenerate graph immersion $M \subset \mathbb A^{n+1}$ with flat affine metric satisfying $\hat\nabla C = 0$ and a point $y \in M$, there exists exactly one graph immersion $M' \subset \mathbb A^{n+1}$ satisfying $\nabla K = 0$ which makes a third-order contact with $M$ at $y$. The immersions $M'$ will be isomorphic for different points $y$. Conversely, given a non-degenerate graph immersion $M' \subset \mathbb A^{n+1}$ satisfying $\nabla K = 0$ and a point $y \in M$, there exists exactly one graph immersion $M \subset \mathbb A^{n+1}$ with flat affine metric and satisfying $\hat\nabla C = 0$ which makes a third-order contact with $M'$ at $y$. The immersions $M$ may be non-isomorphic for different points $y$, however. Note that $M$ is irreducible if and only if $M'$ is irreducible, because both conditions are equivalent to the indecomposability of the metrised algebra $(A,\gamma)$. A similar relation holds for the subclasses of improper affine hyperspheres in these two classes of graph immersions.

The graph immersions with parallel difference tensor which correspond to the Cayley hypersurfaces are given by \cite[eq.~(6.3)]{DillenVrancken98}
\begin{equation} \label{Cayley_siblings}
x_{n+1} = \sum_{d=2}^{n+1} \frac{(-2)^{d-2}}{d!} \sum_{i_1+\dots+i_d = n+1} \prod_{j = 1}^d x_{i_j}.
\end{equation}
These are actually the only improper affine hyperspheres which satisfy $\nabla K = 0$, $K^{n-1} \not= 0$ \cite[Theorem 6.2]{DillenVrancken98}. The similarity between the hypersurfaces \eqref{Cayley_siblings} and the Cayley hypersurfaces was also noted in \cite[eq.~(2)]{HLZ11}. It can be shown that the metrised algebras $(A,\gamma)$ defined by a hypersurface of the form \eqref{Cayley_siblings} at different base points $y$ are isomorphic, and hence \eqref{Cayley_siblings} corresponds to the Cayley hypersurfaces only. In \cite[Theorem 7.1]{DillenVrancken98} it was shown that if an improper affine hypersphere satisfies $\nabla K = 0$, $K^{n-2} \not= 0$, $K^{n-1} = 0$, then it must be the direct sum of a hypersurface \eqref{Cayley_siblings} and a 1-dimensional parabola. It follows that a flat improper affine hypersphere satisfying $\hat\nabla C = 0$, $K^{n-2} \not= 0$, $K^{n-1} = 0$ must be the direct sum of a Cayley hypersurface and a 1-dimensional parabola.

\subsection{Graph immersions with $\nabla C = 0$}

{\theorem Let $f: \Omega \to \mathbb A^{n+1}$ be a non-degenerate graph immersion satisfying $\nabla C = 0$, and let $y \in \Omega$. Let $(A,\gamma)$ be the metrised commutative algebra defined in Theorem \ref{main1a}. Then $f$ has representing function $F$ locally around $y$ given by $F(x) = \frac12\gamma(x,x) - \frac13\gamma(x,x^2)$.

On the other hand, let $(A,\gamma)$ be a metrised commutative algebra. Then the graph of the function $F(x) = \frac12\gamma(x,x) - \frac13\gamma(x,x^2)$ satisfies $\nabla C = 0$, and the difference tensor $K$ consists of the structure coefficients of $A$.

Every graph immersion satisfying $\nabla C = 0$ can be extended to a hypersurface embedding with $\nabla C = 0$, given globally by the graph of a cubic polynomial $F: \mathbb R^n \to \mathbb R$.

The graph immersion is an improper affine hypersphere if and only if $\det\gamma = \pm1$ and the operator $L_x$ is nilpotent for every $x \in A$. }

\begin{proof}
The condition $\nabla C = 0$ is equivalent to $F$ being a cubic polynomial. By \eqref{correspondence} the metrised algebra $(A,\gamma)$ is in one-to-one correspondence with the pair of tensors $(F''(y),F'''(y))$, which implies the first three assertions of the theorem.

At the point $x \in A$ we have $F_{,\alpha\beta}u^{\alpha}v^{\beta} = \gamma(u,v) - 2\gamma(u,x \bullet v)$, and the matrix of the Hessian $F''$ is given by the matrix product $\gamma(I - 2L_x)$. Therefore we have $\det F'' = \det\gamma \cdot \det(I-2L_x)$. By Lemma \ref{IAHS} $f$ then defines an improper affine hypersphere if and only if $\det\gamma = \pm1$ and $\det(I-2L_x) \equiv 1$.

Clearly if $L_x$ is nilpotent, then $\det(I-2L_x) = 1$. Assume now that $\det(I-2L_x) = 1$ for all $x \in A$. Let $\lambda_1,\dots,\lambda_n$ be the eigenvalues of $L_x$. Then we have $\det(I - 2L_{\frac{x}{2t}}) = \prod_{k=1}^n \left(1-\frac{\lambda_k}{t}\right) = 1$ for all $t \not= 0$, and hence $\prod_{k=1}^n (t-\lambda_k) = t^n$. It follows that the $\lambda_k$ are the roots of the polynomial $p(t) = t^n$, and $L_x$ is nilpotent. This completes the proof.
\end{proof}

\bibliography{affine_geometry,jordan,geometry,algebra}

\begin{thebibliography}{10}

\bibitem{BNS90}
Neda Bokan, Katsumi Nomizu, and Udo Simon.
\newblock Affine hypersurfaces with parallel cubic forms.
\newblock {\em T\^ohoku Math. J.}, 42:101--108, 1990.

\bibitem{Bordemann97}
Martin Bordemann.
\newblock Nondegenerate invariant bilinear forms on nonassociative algebras.
\newblock {\em Acta Math. Univ. Comenianae}, 66(2):151--201, 1997.

\bibitem{CS99}
Ivan Correa and Avelino Suazo.
\newblock On a class of commutative power-associative nilalgebras.
\newblock {\em J. Algebra}, 215:412--417, 1999.

\bibitem{DillenVrancken94}
F.~Dillen and L.~Vrancken.
\newblock Calabi-type composition of affine spheres.
\newblock {\em Differential Geom. Appl.}, 4(4):303--328, 1994.

\bibitem{DillenVrancken98}
Franki Dillen and Luc Vrancken.
\newblock Hypersurfaces with parallel difference tensor.
\newblock {\em Japan. J. Math.}, 24(1):43--60, 1998.

\bibitem{Duistermaat01}
Johannes~Jisse Duistermaat.
\newblock On {H}essian {R}iemannian structures.
\newblock {\em Asian J. Math.}, 5:79--91, 2001.

\bibitem{EE06}
M.~Eastwood and V.~Ezhov.
\newblock Cayley hypersurfaces.
\newblock {\em Proc. Steklov Inst. Math.}, 253:241--244, 2006.

\bibitem{ES02b}
Luisa Elgueta and Avelino Suazo.
\newblock Jordan nilagebras of dimension 6.
\newblock {\em Proyecciones}, 21(3):277--289, 2002.

\bibitem{ES02a}
Luisa Elgueta and Avelino Suazo.
\newblock Jordan nilagebras of nilindex $n$ and dimension $n + 1$.
\newblock {\em Commun. Algebra}, 30(11):5547--5561, 2002.

\bibitem{GM75}
Murray Gerstenhaber and Huo~Chul Myung.
\newblock On commutative power-associative nilalgebras of low dimension.
\newblock {\em P. Am. Math. Soc.}, 48:29--32, 1975.

\bibitem{Gigena02}
Salvador Gigena.
\newblock On affine hypersurfaces with parallel second fundamental form.
\newblock {\em T\^ohoku Math. J.}, 54:495--512, 2002.

\bibitem{Gigena03}
Salvador Gigena.
\newblock Classification of five dimensional hypersurfaces with affine normal
  parallel cubic form.
\newblock {\em Contributions to Algebra and Geometry}, 44(2):511--524, 2003.

\bibitem{Gigena11}
Salvador Gigena.
\newblock Inductive schemes for the complete classification of affine
  hypersurfaces with parallel second fundamental form.
\newblock {\em Contributions to Algebra and Geometry}, 52(1):51--73, 2011.

\bibitem{Hildebrand15a}
Roland Hildebrand.
\newblock Centro-affine hypersurface immersions with parallel cubic form.
\newblock {\em Contributions to Algebra and Geometry}, 56(2):593--640, 2015.

\bibitem{HuLi11}
Zejun Hu and Cece Li.
\newblock The classification of 3-dimensional {L}orentzian affine hypersurfaces
  with parallel cubic form.
\newblock {\em Differential Geom. Appl.}, 29(3):361--373, 2011.

\bibitem{HLLV11b}
Zejun Hu, Cece Li, Haizhong Li, and Luc Vrancken.
\newblock The classification of 4-dimensional non-degenerate affine
  hypersurfaces with parallel cubic form.
\newblock {\em J. Geom. Phys.}, 61:2035--2057, 2011.

\bibitem{HLLV11a}
Zejun Hu, Cece Li, Haizhong Li, and Luc Vrancken.
\newblock Lorentzian affine hypersurfaces with parallel cubic form.
\newblock {\em Res. Math.}, 59:577--620, 2011.

\bibitem{HLV08}
Zejun Hu, Haizhong Li, and Luc Vrancken.
\newblock Characterizations of the {C}alabi product of hyperbolic affine
  hyperspheres.
\newblock {\em Res. Math.}, 52:299--314, 2008.

\bibitem{HLV11}
Zejun Hu, Haizhong Li, and Luc Vrancken.
\newblock Locally strongly convex affine hypersurfaces with parallel cubic
  form.
\newblock {\em J. Diff. Geom.}, 87:239--307, 2011.

\bibitem{HLZ11}
Z.J. Hu, C.C. Li, and D.~Zhang.
\newblock A differential geometric characterization of the {C}ayley
  hypersurface.
\newblock {\em Proc. Amer. Math. Soc.}, 139:3697--3706, 2011.

\bibitem{Jacobson68}
Nathan Jacobson.
\newblock {\em Structure and Representation of Jordan Algebras}, volume~39 of
  {\em Colloquium Publications}.
\newblock AMS, 1968.

\bibitem{Koecher99}
Max K\"ocher.
\newblock {\em The Minnesota Notes on Jordan Algebras and their Applications},
  volume 1710 of {\em Lecture Notes in Mathematics}.
\newblock Springer, 1999.

\bibitem{CeceLi14}
Cece Li.
\newblock Affine hypersurfaces with parallel difference tensor relative to
  affine $\alpha$-connection.
\newblock {\em J. Geom. Phys.}, 86:81--93, 2014.

\bibitem{LiZhang15}
Cece Li and Dong Zhang.
\newblock The generalized {C}ayley hypersurfaces and their geometrcal
  characterization.
\newblock {\em Results Math.}, 68:25--44, 2015.

\bibitem{MagidNomizu89}
Martin~A. Magid and Katsumi Nomizu.
\newblock On affine surfaces whose cubic forms are parallel relative to the
  affine metric.
\newblock {\em Proc. Japan Acad. Ser. A}, 65:215--218, 1989.

\bibitem{McCrimmon}
Kevin McCrimmon.
\newblock {\em A taste of {J}ordan algebras}, volume~26 of {\em Universitext}.
\newblock Springer, 2004.

\bibitem{NomizuPinkall89}
Katsumi Nomizu and Ulrich Pinkall.
\newblock Cayley surfaces in affine differential geometry.
\newblock {\em T\^ohoku Math. J.}, 41(4):589--596, 1989.

\bibitem{NomizuSasaki}
Katsumi Nomizu and Takeshi Sasaki.
\newblock {\em Affine Differential Geometry: Geometry of Affine Immersions},
  volume 111 of {\em Cambridge Tracts in Mathematics}.
\newblock Cambridge University Press, Cambridge, 1994.

\bibitem{Schaefer}
Richard~D. Sch\"afer.
\newblock {\em An Introduction to Nonassociative Algebras}.
\newblock Dover, 1996.

\bibitem{Vrancken88}
Luc Vrancken.
\newblock Affine higher order parallel hypersurfaces.
\newblock {\em Ann. Fac. Sci. Toulouse Math.}, 9(3):341--353, 1988.

\bibitem{ZSSS82}
Konstantin~Aleksandrovich Zhevlakov, A.~M. Slin'ko, I.~P. Shestakov, and A.~I.
  Shirshov.
\newblock {\em Rings That Are Nearly Associative}, volume 104 of {\em Pure and
  Applied Mathematics}.
\newblock Acad. Press, 1982.

\end{thebibliography}
\bibliographystyle{plain}

\end{document}